\documentclass[12pt]{amsart}
\usepackage{amsmath,amsthm,amssymb,amscd,url,pict2e}

\makeatletter
\@namedef{subjclassname@2010}{%
  \textup{2010} Mathematics Subject Classification}
\makeatother

\allowdisplaybreaks

\theoremstyle{plain}
\newtheorem{theorem}{Theorem}
\newtheorem{conjecture}[theorem]{Conjecture}
\newtheorem{proposition}[theorem]{Proposition}
\newtheorem{lemma}[theorem]{Lemma}
\newtheorem{corollary}[theorem]{Corollary}

\theoremstyle{definition}
\newtheorem*{definition}{Definition}

\theoremstyle{remark}
\newtheorem{remark}[theorem]{Remark}
\newtheorem{example}[theorem]{Example}
\newtheorem{question}[theorem]{Question}

  {\end{list}}

\newenvironment{parts}[0]{%
  \begin{list}{}%
    {\setlength{\itemindent}{0pt}
     \setlength{\labelwidth}{1.5\parindent}
     \setlength{\labelsep}{.5\parindent}
     \setlength{\leftmargin}{2\parindent}
     \setlength{\itemsep}{0pt}
     }%
   }%
  {\end{list}}
\newcommand{\Part}[1]{\item[\upshape#1]}

\renewcommand{\a}{\alpha}
\renewcommand{\b}{\beta}
\newcommand{\g}{\gamma}
\renewcommand{\d}{\delta}
\newcommand{\e}{\epsilon}
\newcommand{\f}{\phi}
\renewcommand{\l}{\lambda}

\newcommand{\s}{\sigma}

\renewcommand{\t}{\tau}
\newcommand{\z}{\zeta}
\newcommand{\bfmu}{{\boldsymbol\mu}}

\newcommand{\bfxi}{{\boldsymbol\xi}}

\newcommand{\bfzeta}{{\boldsymbol\zeta}}

\newcommand{\F}{\Phi}

\renewcommand{\L}{\Lambda}

\newcommand{\gp}{{\mathfrak{p}}}

\def\Gcal{{\mathcal G}}

\def\Lcal{{\mathcal L}}
\def\Mcal{{\mathcal M}}

\def\Ocal{{\mathcal O}}
\def\Pcal{{\mathcal P}}

\def\Wcal{{\mathcal W}}

\newcommand{\CC}{\mathbb{C}}
\newcommand{\FF}{\mathbb{F}}
\newcommand{\GG}{\mathbb{G}}
\newcommand{\NN}{\mathbb{N}}

\newcommand{\QQ}{\mathbb{Q}}
\newcommand{\RR}{\mathbb{R}}
\newcommand{\TT}{\mathbb{T}}
\newcommand{\ZZ}{\mathbb{Z}}

\def \bfm{{\boldsymbol m}}
\def \bfn{{\boldsymbol n}}

\def \bfz{{\boldsymbol z}}

\def \bfX{{\boldsymbol X}}

\newcommand{\characteristic}{\operatorname{char}}

\newcommand{\dual}{{\scriptscriptstyle\vee}}

\newcommand{\End}{\operatorname{End}}

\newcommand{\Gal}{\operatorname{Gal}}

\newcommand{\Kbar}{{\bar K}}

\newcommand{\LCM}{\operatorname{LCM}}
\newcommand{\Mahler}{{\mathcal M}}  %% Mahler measure

\newcommand{\Norm}{\operatorname{N}}
\newcommand{\Mob}{\operatorname{\mu}}  %% Mobius function

\newcommand{\ord}{\operatorname{ord}}

\newcommand{\Qbar}{{\bar{\QQ}}}
\newcommand{\RA}{\operatorname{\mathcal{R}\!\!\mathcal{A}}} 
\newcommand{\Real}{\operatorname{Re}}

\renewcommand{\setminus}{\smallsetminus}

\newcommand{\Spec}{\operatorname{Spec}}

\newcommand{\tors}{{\textup{tors}}}
\newcommand{\bfzero}{{\boldsymbol0}}

\newcommand{\Zsig}{\operatorname{Zsig}} %% Zsigmondy set
\newcommand{\<}{\langle}
\renewcommand{\>}{\rangle}

\begin{document}

%%%%% To ease editing, for IMPAN journals add:
\baselineskip=17pt
\newcommand{\DATE}{\today}

%%%%%%%%%%%%%%%  Topmatter %%%%%%%%%%%%%%%%%%

\title[Divisor Divisibility Sequences on Tori]{Divisor Divisibility Sequences on Tori}

\author[J.H. Silverman]{Joseph H. Silverman}
\address{Mathematics Department, Box 1917, Brown University,
Providence, RI 02912 USA}
\email{jhs@math.brown.edu}

\date{}

\begin{abstract}
We define the \emph{Divisor Divisibility Sequence} associated to a
Laurent polynomial \(f\in\mathbb{Z}[X_1^{\pm1},\ldots,X_N^{\pm1}]\) to
be the sequence \(W_n(f)=\prod f(\zeta_1,\ldots,\zeta_N)\), where
\(\zeta_1,\ldots,\zeta_N\) range over all \(n\)'th roots of unity with
\(f(\zeta_1,\ldots,\zeta_N)\ne0\).  More generally, we define
\(W_\Lambda(f)\) analogously for any finite subgroup
$\Lambda\subset(\CC^*)^N$.  We investigate divisibility,
factorization, and growth properties of \(W_\Lambda(f)\) as a function
of \(\Lambda\). In particular, we give the complete factorization of
\(W_\Lambda(f)\) when \(f\) has generic coefficients, and we prove an
analytic estimate showing that the rank-of-apparition sets for
\(W_\Lambda(f)\) are not too large.
\end{abstract}

\subjclass[2010]{Primary: 11B39; Secondary: 14G25, 14L99}
\keywords{divisibility sequence}

\maketitle

\tableofcontents

\newpage

%%%%%%%%%%%%%%%%%%%%%%%%%%%%%%%%%%%%%%%%%%%%%%%%%%%%%%%%%%%%%%%%%%%%%%
\section{Introduction}
\label{section:introduction}
%%%%%%%%%%%%%%%%%%%%%%%%%%%%%%%%%%%%%%%%%%%%%%%%%%%%%%%%%%%%%%%%%%%%%%

A classical divisibility sequence is a sequence of (nonzero)
integers $(W_n)_{n\ge1}$ having the property
\begin{equation}
  \label{eqn:clscldiv}
  m\mid n \quad\Longrightarrow\quad W_m\mid W_n.
\end{equation}
Well-known examples of such sequences include $a^n-b^n$, Fibonacci and
Lucas sequences, and elliptic divisibility
sequences. See~\cite{MR1990179} for an overview of the history and
study of divisibility sequences.  These and similar sequences are
associated to multiples of points in one-dimensional algebraic
groups, specifically in (twisted) multiplicative groups or elliptic
curves. They tend to have a number of important properties, such as
those described in Table~\ref{table:dsprops}.

\begin{table}[ht]
\begin{center}
\begin{tabular}{|@{\bfseries\enspace}l|l|} \hline
Divisibility
  &$W_n$ is a divisibility sequence\\ \hline
$\boldsymbol\infty$-Growth
  &$\log|W_n|$ grows like $O(n^d)$ for some $d\ge1$\\ \hline
$\boldsymbol p$-adic Growth
  &$\ord_p(W_n)$ grows regularly (and slowly)\\ \hline
Recursion
  &$W_n$ satisfies a (possibly non-linear) recursion\\ \hline
Zsigmondy
  &Most $W_n$ have a primitive prime divisor\\ \hline
\end{tabular}
\end{center}
\caption{A List of Sequence Properties}
\label{table:dsprops}
\end{table}

It is natural to look for analogous sequences associated to higher
dimensional algebraic groups. An obvious approach (see
Section~\ref{section:history}) yields sequences such as
\begin{equation}
  \label{eqn:Wngcdean1bn1}
  W_n = \gcd(a^n-1,b^n-1)
\end{equation}
for integers~$a$ and~$b$ that are multiplicatively independent
in~$\QQ^*$. Such sequences are quite interesting and lead to deep
theorems and conjectures, for example:
\par\noindent
\textbf{Theorem}.\enspace (Bugeaud, Corvaja, Zannier~\cite{MR1953049})
\[
  \lim_{n\to\infty} \frac{1}{n}\log\gcd(a^n-1,b^n-1)=0.
\]
\par\noindent
\textbf{Conjecture}.\enspace (Ailon, Rudnick~\cite{MR2046966})
\[
  \#\bigl\{n\ge 1 : \gcd(a^n-1,b^n-1)=\gcd(a-1,b-1) \bigr\} = \infty.
\]
\par
In particular, the sequence~\eqref{eqn:Wngcdean1bn1}, which is
associated to the powers of the point~$(a,b)$ in the
group~$\GG_m^2(\QQ)$, fails to have the $\infty$-Growth Property, and
conjecturally fails quite badly.

In this paper we suggest a new way to associate divisibility sequences
to higher dimensional algebraic groups. These sequences have the
Divisibility Property and (conjecturally) the $\infty$-Growth
Property.  For concreteness, in this article we concentrate on the
$N$-dimensional torus $\GG_m^N$. A formulation for more general
algebraic groups is discussed briefly in
Section~\ref{section:ddseqgengp} and will form the content of a
subsequent paper~\cite{DDSpreprint1}.  To define our new sequences, we
replace the point~$(a,b)\in\GG_m^2$ used in~\eqref{eqn:Wngcdean1bn1}
with a divisor in~$\GG_m^N$, or equivalently, with the zero set of a
non-zero Laurent polynomial.

\begin{definition}[Preliminary]
Let~$\bfmu_n\subset\CC^*$ denote the group of $n$'th roots of unity.
The \emph{Divisor Divisibilty Sequence}, or \emph{DD-sequence} for
short, associated to a non-zero Laurent
polynomial~$f\in\ZZ[X_1^{\pm1},\ldots,X_N^{\pm1}]$ is the sequence
\begin{equation}
  \label{eqn:Wnfdef1}
  W_n(f) =
  \prod_{\substack{\z_1,\ldots,\z_N\in\bfmu_n\\ f(\z_1,\ldots,\z_N)\ne0\\}}
  f(\z_1,\ldots,\z_N).
\end{equation}
For example, taking $N=1$ and $f(X)=aX-b$ recovers the classical
divisibility sequence $W_n(aX-b)=a^n-b^n$.
\end{definition}

One easily checks that~$W_n(f)\in\ZZ$ is a divisibility sequence, so
DD-sequences have the Divisibility Property. Further, it is
conjectured (and proven if~$N=1$ or if~$f$ is
``atoral''~\cite{MR3082539})
that~$\log\bigl|W_n(f)\big|\sim n^N\log\Mcal(f)$ as $n\to\infty$,
where~$\Mcal(f)$ is the Mahler measure of~$f$, so DD-sequences
(conjecturally) have the $\infty$-Growth Property.  See
Section~\ref{section:infgrowth} for details.

Computing numerical examples, one quickly notices that higher
dimensional DD-sequences tend to be highly factorizable. We now
explain why, which leads to a generalized definition of the
DD-sequences that are the primary objects of study in this article.

An intrinsic and enlightening way to describe the classical
divisibility property~\eqref{eqn:clscldiv} is to view the positive
integers~$\NN$ as a \emph{partially ordered set} (\emph{poset}),
ordered by divisibility. Then a sequence $W:\NN\to\NN$
satisfies~\eqref{eqn:clscldiv} if and only if it is a morphism of
posets, i.e., a map that preserves the partial ordering. 

Next we observe that a 1-dimensional DD-sequence may be viewed as
assigning an integer~$W_n(f)$ to each finite subgroup~$\bfmu_n$
of~$\CC^*$.  Thus a 1-dimensional DD-sequence may be viewed as a poset
morphism
\[
  \{\text{finite subgroups of $\CC^*$}\} \longrightarrow \NN,
  \quad
  \bfmu_n\longmapsto \bigl|W_n(f)\bigr|,
\]
where we order the subgroups of~$\CC^*$ by inclusion and the elements
of~$\NN$ by divisibility. This suggests that for higher dimensional
DD-sequences, we should define~$W$ to be a function on the set of all
finite subgroups of~$(\CC^*)^N$, rather than restricting attention to
subgroups of the form~$\bfmu_n^N$.

Further, there is no reason to restrict our coefficient ring to
be~$\ZZ$.  Table~\ref{table:notation} sets some notation that will
remain in effect for the rest of this article.  

\begin{table}[ht]
\begin{tabular}{|c@{\quad}l|} \hline
$N$& a positive integer\\[1\jot]
$R$& an integrally closed integral domain\\[1\jot]
$K$& the fraction field of $R$, with fixed separable closure $\Kbar$\\[1\jot] 
$R^{(N)}$& the ring of Laurent polynomials
  $R[X_1^{\pm1},\ldots,X_N^{\pm1}]$\\[1\jot]
$\<\bfzeta\>$& the cyclic subgroup generated by
  $\bfzeta\in\GG_m^N(\Kbar)$\\[1\jot]
$\|\L\|$&the cardinality of a finite subgroup
  $\L\subset\GG_m^N(\Kbar)$ \\[1\jot]
\hline
\end{tabular}
\caption{Notation}
\label{table:notation}
\end{table}

By a slight abuse of
terminology, we view~$R\setminus\{0\}$ as a poset under
divisibility.\footnote{It is really the non-zero ideals that form a
  poset, since $a\mid b$ and $b\mid a$ only imply that $aR=bR$, not
  that $a=b$.} We now define the ``sequences'' that are our primary
object of study.

\begin{definition}
Let~$f\in R^{(N)}$ be a non-zero Laurent polynomial.  The
\emph{Divisor Divisibilty Sequence} (\emph{DD-sequence}) associated
to~$f$ is the map
\[
  W_f : \left\{\begin{tabular}{l}
              finite subgroups\\ of $\GG_m^N(\Kbar)$\\
        \end{tabular}\right\}
   \longrightarrow R, \qquad
  \L\longmapsto \prod_{\substack{\bfzeta\in\L\\ f(\bfzeta)\ne0\\}} f(\bfzeta).
\]
For notational convenience, we may at various times use $W_f(\L)$,
$W_\L(f)$, or $W(f,\L)$ to denote the DD-sequence map, and
when~$\L=\bfmu_n^N$, we may write~$W_n(f)$ or~$W_f(n)$
for~$W_f(\bfmu_n^N)$.
\end{definition}

Our first result provides some justification for this definition.

\begin{proposition}
\label{proposition:Wfdivposet}
Let $f\in R^{(N)}$ be a non-zero Laurent polynomial. 
\begin{parts}
\Part{(a)}
Let $\L\subseteq\GG_m^N(\Kbar)$ be a finite subgroup. 
Then $W_f(\L)\in R$.
\Part{(b)}
The map~$W_f$ is a poset morphism, i.e., 
\[
  \L'\subseteq\L
  \quad\Longrightarrow\quad
  W_f(\L') \bigm| W_f(\L).
\]
\end{parts}
\end{proposition}
\begin{proof}
See Section~\ref{section:defsandprops}.
\end{proof}

How large should we expect~$W_f(\L)$ to be as the size of~$\L$
increases? We observe that~$W_f(\L)$ is a product of~$\|\L\|$ factors,
and that the triangle inequality shows that each factor~$f(\bfzeta)$
is bounded, independently of~$\L$. Thus~$\bigl|W_f(\L)\bigr|$ is likely
to grow exponentially in~$\|\L\|$. If we further take a sequence
of subgroups whose points become equidistributed in the torus
\[
  \TT^N := \bigl\{\bfz=(z_1,\ldots,z_N)\in\CC^N : |z_1|=\cdots=|z_N|=1\bigr\},
\]
then it is natural to compare the growth rate of~$W_f(\L)$ to the
\emph{Mahler measure of~$f$}, which we recall is the quantity
\[
  \Mahler(f) = \exp\left(\int_{\TT^N} \log\bigl|f(z_1,\cdots,z_N)\bigr| \, 
    \frac{dz_1}{z_1} \cdots \frac{dz_N}{z_N}\right);
\]
see~\cite{MR707608,MR1770638}.  For ease of exposition here, we state
the growth conjecture only for the groups~$\L=\bfmu_n^N$; see
Section~\ref{section:infgrowth} for a general formulation.

\begin{conjecture}
\label{conjecture:limMah}
Let $f\in\Qbar^{(N)}\subset\CC^{(N)}$ be a non-zero Laurent polynomial
with algebraic coefficients. Then
\[
  \lim_{n\to\infty} \bigl|W_n(f)\bigr|^{1/n^N} = \Mahler(f).
\]
\end{conjecture}

\begin{theorem}
\label{theorem:infgrowth}
Let $f\in\Qbar^{(N)}\subset\CC^{(N)}$ be a non-zero Laurent polynomial
with algebraic coefficients. 
\begin{parts}
\Part{(a)}
Conjecture~\textup{\ref{conjecture:limMah}} is true if
$N=1$.
\Part{(b)}
Conjecture~\textup{\ref{conjecture:limMah}} is true if~$N\ge2$ and~$f$
is \emph{atoral}, which may be defined by the property that the
intersection of the zero locus~$\{f=0\}$ and the torus~$\TT^N$
satisfies\footnote{The general definition is that an algebraic set
  $X\subseteq\CC^N$ is \emph{atoral} if there exists a non-zero
  regular function~$f$ on~$X$ that vanishes identically on
  $X\cap\TT^N$. And a polynomial~$f$ is atoral if the hypersurface
  $f=0$ is atoral. See~\cite{MR2271943}.}
\begin{equation}
  \label{eqn:atoral}
  \dim \bigl\{ \bfz\in\TT^N : f(\bfz)=0 \bigr\} \le N-2,
\end{equation}
where dim is the dimension as a real analytic subvariety of~$\TT^N$.
\end{parts}
\end{theorem}

The proof of Theorem~\ref{theorem:infgrowth}(a), which we sketch in
Section~\ref{section:infgrowth}, uses a strong estimate for linear
forms in logarithms. Theorem~\ref{theorem:infgrowth}(b), which is due
to Lind, Schmidt, and Verbitskiy~\cite{MR3082539}, applies to ``almost
all''~$f$, since generically for~$N\ge2$, the intersection
$\{f=0\}\cap\TT^N$ has real codimension at least~$2$ in~$\TT^N$.  We
also mention that Conjecture~\ref{conjecture:limMah} is false if~$f$
is allowed to have arbitrary complex coefficients, so any proof will
necessarily require an arithmetic argument; see
Remark~\ref{remark:cmplxcoef} for details.

As noted earlier, DD-sequences tend to be highly factorizable. This is
true even in the classical 1-dimensional setting, since if~$n$ is
highly composite, then~$W_n$ has many factors coming from~$W_m$
for~$m\mid n$. Classically, one generically factors~$W_n$
as~$W_n=\prod_{m\mid n}V_m$, where the factors are defined either
using the M\"obius $\mu$-function or using primitive $n$'th roots.
For ease of exposition, we take the latter approach here, but see
Section~\ref{section:genericfactorization} for both approaches and a
proof of their equivalence.

\begin{proposition}
\label{proposition:WfVffactor}
For each finite subgroup~$\L\subset\GG_m^N(\Kbar)$, let
\[
  V_f(\L) = \prod_{\substack{\text{$\bfzeta$ such that $\<\bfzeta\>=\L$}\\
                 \text{and $f(\bfzeta)\ne0$}\\}}  f(\bfzeta),
\]
so in particular, if~$\L$ is not cyclic, then~$V_f(\L)=1$. Then
\[
  V_f(\L)\in R
  \qquad\text{and}\qquad
  W_f(\L) = \prod_{\L'\subseteq\L} V_f(\L'),
\]
where the product is over all subgroups~$\L'$ of~$\L$.  
%% \begin{parts}
%% \Part{(a)}
%% $V_f(\L)\in R$.
%% \Part{(b)}
%% $\displaystyle W_f(\L) = \prod_{\L'\subseteq\L} V_f(\L')$.
%% where the product is over all subgroups~$\L'$ of~$\L$.  
%% \end{parts}
\end{proposition}

Proposition~\ref{proposition:WfVffactor} gives a generic factorization
of~$W_f(\L)$, but it turns out that~$V_f(\L)$ may generically factor
further, depending on~$\L$ and the non-zero monomials appearing
in~$f$. The main result of Section~\ref{section:genericfactorization}
is Theorem~\ref{theorem:Cfmirred}, which describes the complete
factorization of~$W_f(\L)$ when~$f$ is generic over~$\QQ$ for a
prescribed pattern of non-zero monomials. For such~$f$, we further
show that~$W_f$ is a so-called strong divisibility sequence in the
sense that there is an equality of ideals 
\[
  \gcd\bigl(W_f(\L_1),W_f(\L_2)\bigr) = W_f(\L_1\cap\L_2).
  %% W_f\bigl(\gcd(\L_1,\L_2)\bigr).
\]

We next consider $\gp$-divisibility and~$\gp$-adic behavior of the 
terms in a DD-sequence. This prompts several definitions, which are
generalizations of the classical 1-dimensional case.

\begin{definition}
Let $f\in R^{(N)}$ be a non-zero Laurent polynomial, let~$\gp$ be a
prime ideal of~$R$, and let~$\L\subset\GG_m^N(\Kbar)$ be a finite
subgroup.  Suppose that
\[
   W_f(\L)\in\gp
  \quad\text{and}\quad
   W_f(\L')\notin\gp  \quad\text{for all}\quad \L'\subsetneq\L.
\]
Then we say that~$\gp$ is a \emph{primitive prime divisor
  of~$W_f(\L)$} and that~$\L$ is a \emph{rank of apparition
  for~$\gp$}.  We denote the set of ranks of apparition for~$\gp$ by
\[
  \RA_f(\gp) = \{\L : \text{$\L$ is a rank of apparition for $\gp$}  \}.
\]
It is not hard to prove (Proposition~\ref{proposition:RAfpcyclic})
that~$\RA_f(\gp)$ consists entirely of cyclic groups, so we define the
\emph{Zsigmondy set} of the DD-sequence~$W_f$ to be the set
\[
  \Zsig(f) = \bigl\{\text{cyclic $\L$} : 
              \text{$W_f(\L)$ has no primitive prime divisors} \bigr\}.
\]
\end{definition}

When~$N=1$, it is not hard to show that~$W_f$ has a unique rank
of apparition at~$\gp$, i.e., there is an integer~$r_\gp\ge1$ with the
property that
\[
  \gp\mid W_n(f) \quad\Longleftrightarrow\quad r_\gp\mid n.
\]
But when~$N\ge2$, the set~$\RA_f(\gp)$ may be infinite. The following
analytic result, which is the main theorem of
Section~\ref{section:infgrowth}, shows in particular that~$\RA_f(\gp)$
cannot be too large.  Again, for ease of exposition, we restrict here
to~$R=\ZZ$.

\begin{theorem}
\label{theorem:romrap}
Let $f\in\ZZ^{(N)}$ be a non-zero Laurent polynomial. There is a
constant~$C_f$ so that for all~$\e>0$ we have
\[
  \sum_{p~\text{prime}} \frac{\log p}{p}\sum_{\L\in\RA_f(p)} \frac{1}{\|\L\|^\e}
  \le (N+1)\e^{-1} + C_f.
\]
\end{theorem}

One consequence is the fact that the (cyclic) groups in~$\RA_f(p)$ are
comparatively sparse, since for example the series
$\sum_{\text{$\L$ cyclic}} \|\L\|^s$ diverges for~$\Real(s)<N$, while
Theorem~\ref{theorem:romrap} says that if we restrict
to~$\L\in\RA_f(p)$, then the sum converges for~$\Real(s)>0$.
Theorem~\ref{theorem:romrap} also implies that for any~$\theta>0$, the
(upper logarithmic) Dirichlet density of the set
\[
  \bigl\{ p : 
  \text{$\RA_f(p)$ contains a $\L$ with $\|\L\|<p^\theta$} \bigr\}
\]
is at most~$(N+1)\theta$.  See Section~\ref{section:infgrowth} for
details.

A much-studied classical problem is that of perfect powers, and more
recently, powerful numbers, in divisibility sequences. For example, it
is known that the only perfect powers in the Fibonacci sequence
are~$1$,~$8$, and~$144$, a longstanding conjecture recently proven by
Bugeaud, Mignotte,and Siksek~\cite{MR2215137}.  The analogous question
of classifying powerful Fibonacci numbers is still open. (We recall
that~$n$ is \emph{powerful} if, whenever~$p\mid n$, then~$p^2\mid n$.)
More generally, combining a result of Shorey and
Stewart~\cite{MR915504} with Siegel's theorem on integral points on
curves gives an (ineffective) proof that there are only finitely many
perfect powers in any non-degenerate binary recurrent sequence.

This suggests the general question of describing perfect powers and
powerful numbers in higher dimensional DD-sequences. We do not
consider such questions in this paper, but we note that some care must
be taken, because if the Laurent polynomial~$f$ has symmetries, then
the associated DD-sequence is often divisible by large powers. We
illustrate this principle in Section~\ref{section:ddseqsympoly} by
studying the DD-sequence for the family 
\[
  P_T(X,Y)=X+X^{-1}+Y+Y^{-1}+T\in \ZZ[T]^{(2)},
\]
so~$W_n(P_T)\in\ZZ[T]$ is a polynomial of
degree~$n^2$. We prove that~$W_n(P_T)$ is almost a perfect~$8$'th
power; more precisely, it factors in~$\ZZ[T]$
as $W_n(P_T)=A_n(T)B_n(T)^8$ with $\deg B_n(T)=\frac18n^2+O(n)$.  We
also prove that~$W_n(P_{2T+4})$ and~$W_n(P_T)$ have a common factor
in~$\ZZ[T]$ of degree roughly~$2n$.

In summary, there are many natural questions and problems associated
to DD-sequences, some of which are direct analogues of the
one-dimensional situation, and some of which appear only in the
higher-dimen\-sion\-al setting.  And while some of these questions
have elementary answers, others appear to lead to deep and interesting
conjectures. In this article we give some elementary results,
state some conjectures as motivation for the study of
DD-sequences, and prove two deeper theorems:
\begin{parts}
\Part{\textbullet}
Generic factorization of DD-sequences, covered in
Sections~\ref{section:genericfactorization}
and~\ref{section:strongdivgeneric}; see especially
Theorems~\ref{theorem:WfVfprodCf} and~\ref{theorem:Cfmirred} and
Proposition~\ref{proposition:genericstrongdiv}.
\Part{\textbullet}
Distribution of ranks of apparition, covered in
Section~\ref{section:rankofapp}; see especially
Theorem~\ref{theorem:sumlogpp} and
Corollary~\ref{corollary:densityRApf}.
\end{parts}

\par\noindent\textbf{Addendum}.\
Recent preprints by Habegger~\cite{ArXiV:1608.04547v1} and
Dimitrov~\cite{Dimitrov2016} include proofs of
Conjecture~\ref{conjecture:limMah}. The two papers use distinct
methods, and the specific estimates that they prove are rather
different, but both suffice to prove
Conjecture~\ref{conjecture:limMah}.  On the other hand, neither method
seems to be strong enough to prove the growth conjecture for 
general algebraic subgroups of~$\GG_m^N$ as described in
Conjecture~\ref{conjecture:growth}.

%%%%%%%%%%%%%%%%%%%%%%%%%%%%%%%%%%%%%%%%%%%%%%%%%%%%%%%%%%%%%%%%%%%%%%
\section{Some Brief Remarks on Divisibility Sequences}
\label{section:history}
%%%%%%%%%%%%%%%%%%%%%%%%%%%%%%%%%%%%%%%%%%%%%%%%%%%%%%%%%%%%%%%%%%%%%%
Factorization and other properties of sequences $a^n-b^n$ and the
Fibonacci and Lucas sequences have been studied for a very long time,
so we will not attempt to give a history. The arithmetic of elliptic
divisibility sequences (EDS) was first seriously studied
Ward~\cite{MR0023275} and has since attracted considerable
attention. Again, the literature is too vast to survey here.

The first reference of which we are aware for higher degree, but still
one-dimensional, DD-sequences, is the~1916 Ph.D. thesis of
T.\ Pierce~\cite{MR1503584}. He takes a monic polynomial
$f(X)\in\ZZ[X]$, factors it (over~$\CC$) as $f(X)=\prod(X-\a_i)$, and
studies elementary arithmetic properties of the associated
1-dimensional DD-sequence $W_n(f) = \prod_{i=1}^d (1-\a_i^n)$.  In
particular, he gives various factorizations of~$W_n(f)$ and studies
the relationship between divisors of~$W_n(f)$ and roots
of~$f(X)\equiv0\pmod{n}$, especially when~$n$ is prime or a prime
power.

For higher dimensional divisibility sequences, we have already
mentioned the interesting sequences~$\gcd(a^n-1,b^n-1)$ investigated
in~\cite{MR2046966,MR1953049}, and there are analogous sequences on
abelian varieties, for example the gcd of two EDS~\cite{MR2081943},
but they do not appear to have the growth property. A general
``non-growth'' theorem for sequences of this sort, conditional on
Vojta's conjecture, is given in~\cite{MR2162351}.

Marco Streng~\cite{MR2377368} has studied an interesting
generalization of~EDS in the case that the elliptic curve~$E$ has
complex multiplication. He associates to a point~$P\in E(K)$ a
``sequence'' indexed by the elements of the endomorphism
ring,~$\a\in\End(E)$, more-or-less by taking the (square root of the)
denominator of~$x\bigl(\a(P)\bigr)$. He proves the divisibility
property and a Zsigmondy theorem regarding primitive prime divisors.

Although somewhat different, we must also mention Stange's theory of
elliptic nets~\cite{MR2833790}. This generalization of classical EDS
attaches a ``sequence'' indexed by~$\ZZ^r$ to a collection of linearly
independent points~$P_1,\ldots,P_r$ on an elliptic curve. She proves,
among many results, that the terms in an elliptic net are generated by
a non-linear recursion applied to a finite (but potentially quite
large) set of initial values.

%%%%%%%%%%%%%%%%%%%%%%%%%%%%%%%%%%%%%%%%%%%%%%%%%%%%%%%%%%%%%%%%%%%%%%
\section{Basic Properties of DD-Sequences}
\label{section:defsandprops}
%%%%%%%%%%%%%%%%%%%%%%%%%%%%%%%%%%%%%%%%%%%%%%%%%%%%%%%%%%%%%%%%%%%%%%
We begin with the elementary proof of
Proposition~\ref{proposition:Wfdivposet}, where we note the importance
in the proof of our assumption that~$R$ is an integrally closed
integral domain.

\begin{proof}[Proof of Proposition \textup{\ref{proposition:Wfdivposet}}]
(a)\enspace
The finite subgroup~$\L$ of~$\GG_m^N(\Kbar)$
and the set $\{\bfz\in\GG_m^N(\Kbar) : f(\bfz)=0\}$ are $\Gal(\Kbar/K)$
invariant, so $W_f(\L)$ is Galois invariant, and hence $W_f(\L)\in K$.
On the other hand, every root of unity is integral over~$R$,
so~$W_f(\L)$ is integral over~$R$. But by assumption, the ring~$R$
is integrally closed, hence~$W_f(\L)\in R$.
\par\noindent(b)\enspace
From~(a) we know that~$W_f(\L)/W_f(\L')$ is in~$K$.  Further, the
inclusion $\L'\subseteq\L$ implies that the quotient
\[
  \frac{W_f(\L)}{W_f(\L')}
  = \prod_{\bfzeta\in\L\setminus\bigl(\L'\cup\{f=0\}\bigr)} f(\bfzeta)
\]
is integral over~$R$. Again the fact that~$R$ is integrally closed
tells us that the quotient is in~$R$.
\end{proof}

We next consider the factorization of a DD-sequence, analogous to the
classical factorization of~$X^n-1$ as a product of cyclotomic
polynomials. This latter factorization may be described either using
primitive $n$'th roots of unity or via the classical M\"obius
function.  More generally, we note that there is a M\"obius function
attached to any (locally finite) poset~\cite[Section~8.6]{MR0356989},
so in particular there is a M\"obius function associated to the set of
finite subgroups of~$\GG_m^N(\Kbar)$, ordered by inclusion. We denote
this function by
\begin{equation}
  \label{eqn:mobiusfinsubgp}
  \Mob : \{\text{pairs of finite subgroups
    $\L'\subseteq\L\subset\GG_m^N(\Kbar)$}\} \longrightarrow \ZZ.
\end{equation}
It is characterized by $\Mob(\L,\L)=1$ and the M\"obius inversion
formula.
%% \[
%%   \sum_{\L'\subseteq\L} \Mob(\L,\L') = \begin{cases}
%%      1&\text{if $\L=\L'$,} \\
%%      0&\text{if $\L\ne\L'$.} \\
%%   \end{cases}
%% \].

\begin{theorem}
\label{theorem:WfprodVf}
Let $f\in R^{(N)}$ be a non-zero Laurent polynomial, and
let~$\L\subset\GG_m^N(\Kbar)$ be a finite group.
\begin{parts}
\Part{(a)}
The following formula gives two equivalent ways to define a
quantity~$V_f(\L)$\textup:
\[
  V_f(\L) := \prod_{\substack{\bfzeta\in\GG_m^N(\Kbar)~\text{such that}\\
                  f(\bfzeta)\ne0~\text{and}~\<\bfzeta\>=\L\\}} f(\bfzeta)
   = \prod_{\L'\subseteq\L} W_f(\L')^{\Mob(\L,\L')}.
\]
In particular, if~$\L$ is not cyclic, then~$V_f(\L)=1$.
\Part{(b)}
$V_f(\L)\in R$.
\Part{(c)}
$W_f(\L)$ factors in $R$ as
\[
  W_f(\L) = \prod_{\L'\subseteq\L} V_f(\L').
\]
\Part{(d)}
Let $\bfxi\in\GG_m^N(\Kbar)$ have order~$n$. Then
\[
  V_f\bigl(\<\bfxi\>\bigr) = \prod_{\substack{\text{$d\in(\ZZ/n\ZZ)^*$ such}\\
    \text{that $f(\bfxi^d)\ne0$}\\}} f(\bfxi^d)
  \quad\text{and}\quad
  W_f\bigl(\<\bfxi\>\bigr) = \prod_{\substack{\text{$d\in\ZZ/n\ZZ$ such}\\
    \text{that $f(\bfxi^d)\ne0$}\\}} f(\bfxi^d).
\]
\end{parts}
\end{theorem}
\begin{proof}
(a)\enspace
We start with the formula for~$V_f(\L)$ in terms of the M\"obius function
and derive the formula in terms of generators for~$\L$. 
\begin{align*}
   \prod_{\L'\subseteq\L} W_f(\L')^{\Mob(\L,\L')}
   &= \prod_{\L'\subseteq\L} 
      \biggl(\prod_{\bfzeta\in\L'\setminus\{f=0\}} f(\bfzeta)\biggr)^{\Mob(\L,\L')} \\
   &= \prod_{\bfzeta\in\GG_m^N(\Kbar)_\tors\setminus\{f=0\}}
       \biggl(
         \prod_{\substack{\L'~\text{such that}\\ 
            \<\bfzeta\>\subseteq\L'\subseteq\L\\}}
        f(\bfzeta)^{\Mob(\L,\L')} \biggr).
\end{align*}
M\"obius inversion tells us that for any $\L_1\subseteq\L_2$, we have
\[
  \sum_{\L_1\subseteq\L'\subseteq\L_2} \Mob(\L_2,\L') = 
    \begin{cases}
         1&\text{if $\L_1=\L_2$,} \\
         0&\text{otherwise.} \\
    \end{cases}.
\]
Hence
\[
  \prod_{\L'\subseteq\L} W_f(\L')^{\Mob(\L,\L')}
  =  \prod_{\substack{\bfzeta\in\GG_m^N(\Kbar)_\tors\setminus\{f=0\}\\
         \text{such that}~\<\bfzeta\>=\L\\}} f(\bfzeta),
\]
which is the desired formula.
\par\noindent(b)\enspace
We know from Proposition~\ref{proposition:Wfdivposet}(a)
that~$W_f(\L)\in R$, so the formula for $V_f(\L)$ as a product of
(positive and negative) powers of~$W_f(\L')$ shows that~$V_f(\L)\in
K$. On the other hand, the formula for~$V_f(\L)$ as a product of
values of~$f(\bfzeta)$ for $N$-tuples
of roots of unity~$\bfzeta\in\GG_m^N(\Kbar)_\tors$ shows that~$V_f(\L)$
is integral over~$R$. Hence~$V_f(\L)\in R$ by our assumption that~$R$ is
integrally closed.
\par\noindent(c)\enspace
This is just M\"obius inversion, but we do the calculation. We have
\begin{align*}
  \prod_{\L'\subseteq\L} V_f(\L')
  &= \prod_{\L'\subseteq\L} 
     \prod_{\L''\subseteq\L'} W_f(\L'')^{\Mob(\L',\L'')} \\
  &= \prod_{\L''\subseteq\L} \prod_{\L''\subseteq\L'\subseteq\L} 
         W_f(\L'')^{\Mob(\L',\L'')} \\
  &= W_f(\L)\quad\text{from M\"obius inversion.}
\end{align*}
\par\noindent(d)\enspace
The first formula is immediate from~(a) applied to the cyclic
group $\L=\<\bfzeta\>$, and then the second formula follows from
the first formula and the factorization of~$W_f(\L)$ given in~(c).
\end{proof}

%%%%%%%%%%%%%%%%%%%%%%%%%%%%%%%%%%%%%%%%%%%%%%%%%%%%%%%%%%%%%%%%%%%%%%
\section{Generic Factorization of DD-Sequences}
\label{section:genericfactorization}
%%%%%%%%%%%%%%%%%%%%%%%%%%%%%%%%%%%%%%%%%%%%%%%%%%%%%%%%%%%%%%%%%%%%%%

Theorem~\ref{theorem:WfprodVf}(c) gives a generic factorization
of~$W_f(\L)$ in~$R$ that is analogous to the factorization of~$X^n-1$
as a product of cyclotomic polynomials, but it turns out
that~$W_f(\L)$ may admit a further generic factorization, depending on
the interaction of~$\L$ with the non-zero monomials appearing in~$f$.
In this section we describe this factorization and prove that
if~$K\cap\Qbar=\QQ$ and the non-zero coefficients of~$f$ are
algebraically independent over~$\QQ$, then~$W_f(\L)$ does not factor
further. This last result is a DD-sequence analogue of the
irreducibility of the cyclotomic polynomials over~$\QQ$.  And in the
next section (Proposition~\ref{proposition:genericstrongdiv}) we use
these results to show that a generic DD-sequence satisfies a strong
divisibility property given by an equality of ideals\footnote{This
  generalizes the classical \emph{strong divisibility property}
  $\gcd(W_m,W_n)=W_{\gcd(m,n)}$ satisfied, for example, by Fibonacci
  and Lucas sequences.}
\[
  \gcd\nolimits_R\bigl(W_{f}(\L_1), W_{f}(\L_2)\bigr)
  = W_{f}(\L_1\cap\L_2)R.
\]

We begin with some useful notation. 
In order to write elements of~$R^{(N)}$ succinctly, for~$N$-tuples 
\[
  \bfm=(m_1,\ldots,m_N)\in\ZZ^N
  \quad\text{and}\quad
  \bfX=(X_1,\ldots,X_N),
\]
we let
\[
  \bfX^\bfm = X_1^{m_1}\cdots X_N^{m_N}.
\]
Then every~$f\in R^{(N)}$ can be written as
\[
  f(\bfX) = \sum_{\bfm\in\ZZ^N} a_\bfm(f) \bfX^{\bfm}
\]
with $a_\bfm(f)\in R$ and all but finitely many $a_\bfm(f)=0$.  We
note that the unit group of $R^{(N)}$ is exactly the set of monomials
with unit coefficients, i.e., $\{ u\bfX^\bfm : \bfm\in\ZZ^N,\,u\in
R^*\}$.  As usual, we say that an element~$f\in R^{(N)}$ is
\emph{irreducible} if it is not a unit and has no
factorizations~$f=gh$ except with~$g$ or~$h$ a unit.

\begin{example}
\label{example:Vfsquare}
We give an example illustrating the fact that~$W_f(\L)$ may admit a
further generic factorization beyond its factorization as a product
of~$V_f(\L')$ values as in Theorem~\ref{theorem:WfprodVf}(c).  Let
$f(X)=aX^2-b$ with~$a$ and~$b$ independent indeterminates. Then
\[
  V_f(\bfmu_n)
  = \prod_{\zeta~\text{with}~\<\zeta\>=\mu_n} (a\z^2-b)
  =\begin{cases}
     \F_n(a,b)&\text{if $n$ is odd,} \\
     \F_{n/2}(a,b)^2&\text{if $n$ is even,} \\
  \end{cases}
\]
where $\F_n(U,V)\in\ZZ[U,V]$ is the homogenized $n$'th cyclotomic
polynomial. 
%% \begin{align*}
Thus if~$n$ is even, then~$V_f(\bfmu_n)$ is generically a square. This
is due to the fact that~$f(X)$ is a polynomial in~$X^2$.
\end{example}

Our next result generalizes Example~\ref{example:Vfsquare} to all
DD-sequences, but first we need some additional notation.

\begin{definition}
Let~$f\in R^{(N)}$.  The \emph{set of monomials of $f$} is
\[
  M(f) = \bigl\{\bfm\in\ZZ^N : a_\bfm(f)\ne\bfzero\}.
\]
For any finite subset $M\subset\ZZ^N$ and any
$\bfxi\in\GG_m^N(\Kbar)_\tors$, we let
\begin{align*}
  \bfxi^M &= \{\bfxi^\bfm : \bfm\in M\},\\
  K(\bfxi^M) &= \text{\textup(the field generated by $\bfxi^M$\textup).}
\end{align*}
We note that~$K(\bfxi^M)$ is a Galois extension of~$K$, since even in
the case that~$K$ has positive characteristic, adjoining roots of
unity gives a separable extension. Further,~$\Gal\bigl(K(\bfxi^M)/K\bigr)$
is abelian.
\end{definition}

\begin{theorem}
\label{theorem:WfVfprodCf}
For $f\in R^{(N)}$ and $\bfxi\in\GG_m^N(\Kbar)_\tors$, define
\[
  C_f(\bfxi) = \prod_{\t\in\Gal(K(\bfxi^{M(f)})/K)} \t\bigl(f(\bfxi)\bigr).
\]
\vspace{-10pt}
\begin{parts}
\Part{(a)}
$C_f(\bfxi)\in R$.
\Part{(b)}
Let~$\L\subset\GG_m^N(\Kbar)_\tors$ be a cyclic group,
and let~$\bfxi_1,\ldots,\bfxi_r$ be generators for~$\L$
that are representatives for the distinct
orbits for $\Gal(\Kbar/K)$ acting on the set of all generators of~$\L$.
Then~$V_f(\L)$ factors in~$R$ as
\[
  V_f(\L) 
    = \prod_{i=1}^r  C_f(\bfxi_i)^{[K(\bfxi_i):K(\bfxi_i^{M(f)})]}.
\]
\Part{(c)}
Let $\L\subset\GG_m^N(\Kbar)$ be a finite group. Then~$W_f(\L)$
factors in~$R$ as
\[
  W_f(\L) = \prod_{\substack{\text{cyclic sub-}\\
       \text{groups}~\L'\subset\L\\}} 
     \prod_{\substack{\text{generators $\bfxi$ for $\L'$ lying}\\
             \text{in distinct $\Gal(\Kbar/K)$-orbits}\\}}
          C_{f}(\bfxi)^{[K(\bfxi):K(\bfxi^{M(f)})]}.
\]
\end{parts}
\end{theorem}

\begin{proof}
(a) The quantity~$f(\bfxi)$ is in~$K(\bfxi^{M(f)})$, so the product of
all of the $K(\bfxi^{M(f)})/K$-conjugates of~$f(\bfxi)$ is in~$K$, and
hence $C_f(\bfxi)\in K$.  On the other hand,
each~$\t\bigl(f(\bfxi)\bigr)$ is integral over~$R$, so~$C_f(\bfxi)$ is
integral over~$R$. Hence~$C_f(\bfxi)\in R$ from our assumption that~$R$
is integrally closed.
\par\noindent(b)\enspace
The coordinates of any element~$\bfxi\in\GG_m^N(\Kbar)_\tors$ are
roots of unity, so for all~$\s\in\Gal(\Kbar/K)$, we
have~$\s(\bfxi)=\bfxi^k$ for some~$k=k(\s)$. It follows
that~$\Gal(\Kbar/K)$ acts on~$\L$, and similarly it acts on the set of
generators of~$\L$.  This justifies our choice of representatives for
the orbits. Further, the action factors through an abelian group,
since roots of unity generate abelian extensions.  We now compute,
where for notational convenience we assume that all products omit any
factors that vanish.
\begin{align*}
  V_f(\L)
  &= \prod_{\text{generators $\bfxi$ for $\L$}} f(\bfxi)
     &&\text{from Theorem~\ref{theorem:WfprodVf}(a),} \\
  &= \prod_{i=1}^r \prod_{\l\in\Gal(K(\bfxi_i)/K)}  f\bigl(\l(\bfxi_i)\bigr) 
     &&\text{since $\bfxi_1,\ldots,\bfxi_r$ are orbit reps,} \\
  &= \prod_{i=1}^r 
       \biggl( \prod_{\substack{\s\in\Gal(K(\bfxi_i)/K(\bfxi_i^{M(f)}))\\
                                 \t\in\Gal(K(\bfxi_i^{M(f)})/K)\\}}
                 \s\t\bigl(f(\bfxi_i)\bigr)\biggr) \hidewidth \\
  &= \prod_{i=1}^r 
      \biggl( \prod_{\t\in\Gal(K(\bfxi_i^{M(f)})/K)}
                \t\bigl(f(\bfxi_i)\bigr)\biggr)^{[K(\bfxi):K(\bfxi^{M(f)})]} 
       \hidewidth \\
  &\multispan3\hfill
     \text{since $\Gal(K(\bfxi_i)/K(\bfxi_i^{M(f)}))$ fixes $f(\bfxi_i)$,} \\
  &= \prod_{i=1}^r C_f(\bfxi_i)^{[K(\bfxi):K(\bfxi^{M(f)})]} 
    &&\text{by definition of $C_f$.}
\end{align*}
\par\noindent(c)\enspace
This follows from~(b) and the decomposition of~$W_f(\L)$ as a product
of~$V_f(\L')$ values in Theorem~\ref{theorem:WfprodVf}(c).
\end{proof}

Theorem~\ref{theorem:WfVfprodCf}(c) gives a factorization of~$W_f(\L)$
in~$R$ that is a product of powers of~$C_f(\bfxi)$ values. For
particular choices of~$R$,~$f$, and~$\L$, it is quite possible for
these~$C_f(\bfxi)$ to factor further. The main result of this section
says that if the coefficients of~$f$ are generic for the given
pattern~$M(f)$ of non-zero coefficients, then the~$C_f(\bfxi)$
appearing in Theorem~\ref{theorem:WfVfprodCf}(c) are irreducible
in~$R$. In particular, combining Theorem~\ref{theorem:Cfmirred} with
Corollary~\ref{corollary:WfGfactoroverQ} gives a generalization to
DD-sequences of the classical result that the complete factorization
of~$X^n-1$ in~$\QQ[X]$ is as a product of cyclotomic polynomials.

\begin{theorem}
\label{theorem:Cfmirred}
Let $\FF$ be a field, let~$M\subset\ZZ^N$ be a finite set
with~$\bfzero\in M$, let~$R$ be the polynomial ring
\[
  R = \FF[a_\bfm]_{\bfm\in M},
\]
where the~$a_\bfm$ are independent indeterminates, and let~$f_M\in
R^{(N)}$ be the Laurent polynomial
\[
  f_M(\bfX)=\sum_{\bfm\in M} a_\bfm\bfX^{\bfm}\in R^{(N)}.
\]
Thus~$f_M$ is the generic Laurent polynomial over~$\FF$
whose non-zero monomials are in the positions specified by~$M$.
\begin{parts}
\Part{(a)}
For all $\bfxi\in\GG_m^N(\Kbar)_\tors$, the element
\begin{equation}
  \label{eqn:CfMxidefx}
  C_{f_M}(\bfxi) := \prod_{\t\in\Gal(\FF(\bfxi^{M})/\FF)} \t\bigl(f_M(\bfxi)\bigr) \in R
\end{equation}
described in Theorem~$\ref{theorem:WfVfprodCf}$ is irreducible in~$R$.
\Part{(b)}
The formula in Theorem~\textup{\ref{theorem:WfVfprodCf}(c)} is a
complete factorization of $W_{f_M}(\L)$ into irreducible elements
of~$R$.
\end{parts}
\end{theorem}

\begin{remark}
\label{remark:0notinMx}
In the setting of Theorem~\ref{theorem:Cfmirred}, if we drop the
condition that $\bfzero\in M$, then it is possible for
$C_{f_M}(\bfxi)$ to be reducible. For example, take~$\FF=\QQ$ let
$\bfm,\bfn\in\ZZ^N$ be exponents with~$\bfxi^\bfm\ne\pm1$
and~$\bfxi^\bfn=1$, and
let~$M=\{\bfm,\bfm+\bfn\}$. Then~$\bfxi^M=\{\bfxi^\bfm\}$ consists of
a single element, and we have
\begin{align*}
  C_{f_M}(\bfxi)
  &= \prod_{\t\in\Gal(\QQ(\bfxi^M)/\QQ)} \t\bigl(f_M(\bfxi)\bigr) \\*
  &= \prod_{\t\in\Gal(\QQ(\bfxi^\bfm)/\QQ)} \t\bigl(f_M(\bfxi)\bigr) \\
  &= \prod_{\t\in\Gal(\QQ(\bfxi^\bfm)/\QQ)} \t\bigl(f_M(1)\bfxi^\bfm\bigr) 
    \quad\text{since $M=\{\bfm,\bfm+\bfn\}$,}\\*
  &= f_M(1)^{[\QQ(\bfxi^\bfm):\QQ]} \Norm_{\QQ(\bfxi^\bfm)/\QQ} \bfxi^\bfm 
  = \pm f_M(1)^{[\QQ(\bfxi^\bfm):\QQ]}.
\end{align*}
Since~$\bfxi^\bfm\ne\pm1$ by assumption, we have $\bfxi^\bfm\notin\QQ$,
so~$C_{f_M}(\bfxi)$ is reducible.
For an explicit example, we take $N=1$,
$M=\{1,5\}$,~$f(X)=a_1X+a_5X^5$, and $\bfxi=i=\sqrt{-1}$.
Then $\bfxi^M=\{i,i^5\}=\{i\}$ and
\[
  C_{f_M}(\bfxi)
  = \prod_{\hidewidth\t\in\Gal(\QQ(i)/\QQ)\hidewidth} \t(a_1i+a_5i^5)
  = (a_1i+a_5i)(-a_1i-a_5i)
  = (a_1+a_5)^2.
\]
\end{remark}

\begin{proof}[Proof of Theorem $\ref{theorem:Cfmirred}$]
Each factor $\t\bigl(f_M(\bfxi)\bigr)$ in the
product~\eqref{eqn:CfMxidefx} that defines $C_{f_M}(\bfxi)$ is a
non-trivial homogeneous linear form in the indeterminates $a_\bfm$, and
such linear forms are irreducible in the polynomial
ring \text{$R\otimes_\FF\FF(\bfxi^M)$}, which is a UFD. Hence any
non-constant factor of $C_{f_M}(\bfxi)$ in~$R$ has the form
\begin{equation}
  \label{eqn:Hprodx3}
  \b \prod_{\t\in H} \t\bigl(f_M(\bfxi)\bigr) \in R
\end{equation}
for some $\b\in\FF(\bfxi^M)^*$ and some non-empty subset
\[
  H\subseteq \Gal\bigl(K(\bfxi^M)/K\bigr) = \Gal(\FF(\bfxi^M)/\FF).
\]

We now use the assumption that $\bfzero\in M$, which implies
that~$f_M(\bfxi)$ has an~$a_{\bfzero}$ term. The elements of~$H$ act 
trivially on~$a_\bfzero$, so~\eqref{eqn:Hprodx3}
has a monomial of the form $\b a_0^{\#H}$. But~\eqref{eqn:Hprodx3} is
in~$R$, i.e., its coefficients are in~$\FF$, so~$\b\in\FF^*$.  

Next we apply an arbitrary element $\s\in\Gal(\FF(\bfxi^M)/\FF)$ to
the product~\eqref{eqn:Hprodx3}. By assumption, this leaves the
product invariant, so again by unique factorization in
\text{$R\otimes_\FF\FF(\bfxi^M)$} and the fact that the homogeneous
linear forms~$f_M(\xi)$ are irreducible, we deduce that for all
$\s\in\Gal(\FF(\bfxi^M)/\FF)$ and all $\tau\in H$ there is a
$\l_{\s,\t}\in H$ and a scalar $\g_{\s,\t}\in\FF(\bfxi^M)^*$ such that
\begin{equation}
  \label{eqn:sfMtcfM5}
  \s\bigl(\t\bigl(f_M(\bfxi)\bigr)\bigr) 
  = \g_{\s,\t}\cdot \l_{\s,\t}\bigl(f_M(\bfxi)\bigr).
\end{equation}
Again using the assumption that $\bfzero\in M$, we look at
the~$a_{\bfzero}$ monomial on both sides of~\eqref{eqn:sfMtcfM5}. This
monomial is unaffected by the action of Galois, which allows us to
conclude that~$\g_{\s,\t}=1$. Expanding~\eqref{eqn:sfMtcfM5} gives
\[
  \sum_{\bfm\in M} \s\t(\bfxi^\bfm)a_\bfm = \sum_{\bfm\in M} \l_{\s,\t}(\bfxi^\bfm)a_\bfm.
\]
Keeping in mind that the~$a_\bfm$ are indeterminates, we find that
\[
  \s\t(\bfxi^\bfm) = \l_{\s,\t}(\bfxi^\bfm)
  \quad\text{for all $\bfm\in M$,}
\]
and then, since $\s,\t,\l_{\s,\t}\in\Gal(\FF(\bfxi^M)/\FF)$, we
conclude that
\[
  \s\t=\l_{\s,\t}.
\]
The key here is the fact that~$\t$ and~$\l_{\s,\t}$ are in~$H$,
while~$\s$ is an arbitrary element of~$\Gal(\FF(\bfxi^M)/\FF)$. Hence
for any $g\in\Gal(\FF(\bfxi^M)/\FF)$ and any $h\in H$, we can take
$\s=gh^{-1}$ and $\t=h$ to conclude that
\[
  g = gh^{-1}h = \l_{gh^{-1},h} \in H.
\]
This proves that $H=\Gal(\FF(\bfxi^M)/\FF)$, and thus that the
product~\eqref{eqn:Hprodx3} is equal to~$C_{f_M}(\bfxi^\dual)$ up to
multiplication by an element of~$\FF^*$. This proves~(a), and~(b) is
immediate from~(a) and Theorem~\textup{\ref{theorem:WfVfprodCf}(c)}.
\end{proof}

\begin{corollary}
\label{corollary:WfGfactoroverQ}
Assume that~$\characteristic(K)=0$ and that~$K\cap\Qbar=\QQ$.
Let $f\in R^{(N)}$, and let~$\bfxi\in\GG_m^N(\Kbar)_\tors$.
\vspace{2\jot}
\begin{parts}
\Part{(a)}
\quad$\displaystyle
  V_f\bigl(\<\bfxi\>\bigr)= C_f(\bfxi)^{[\QQ(\bfxi):\QQ(\bfxi^{M(f)})]}$.
\vspace{2\jot}
\Part{(b)}
\quad$\displaystyle
  W_f\bigl(\<\bfxi\>\bigr) = \prod_{\hidewidth\substack{\text{cyclic sub-}\\ 
            \text{groups}~\<\bfzeta\>\subset\<\bfxi\>\\}\hidewidth} 
          C_{f}(\bfzeta)^{[K(\bfzeta):K(\bfzeta^{M(f)})]}$.
\end{parts}
Further, if~$f$ is generic for its pattern of non-zero coefficients as
in Theorem~$\ref{theorem:Cfmirred}$, then~\textup{(a)}
and~\textup{(b)} give the complete factorizations of~$V_f$ and~$W_f$
in~$R$.
\end{corollary}

The proof of Corollary~\ref{corollary:WfGfactoroverQ} uses
Theorem~\ref{theorem:WfVfprodCf} and the following transitivity
result, which is more-or-less equivalent to the irreducibility of the
cyclotomic polynomials over~$\QQ$.

\begin{lemma}
\label{lemma:cycsubx}
Assume that~$\characteristic(K)=0$ and~$K\cap\Qbar=\QQ$.  Let
$\bfzeta_1,\bfzeta_2\in\GG_m^N(\Kbar)_\tors$. Then the following are
equivalent\textup:
\begin{parts}
\Part{(a)}
The points $\bfzeta_1$ and $\bfzeta_2$ generate the same cyclic subgroup
of~$\GG_m^N(\Kbar)$.
\Part{(b)}
There is a $\s\in\Gal(\Kbar/K)$ such that $\bfzeta_2=\s(\bfzeta_1)$.
\end{parts}
\end{lemma}
\begin{proof}
Let $\L\subset\GG_m^N(\Kbar)$ be a finite subgroup, let~$d=\#\L$,
let~$\z_d\in\Kbar$ be a primitive $d$'th root of unity, and let
$\s\in\Gal(\Kbar/K)$. Then~$\s(\z_d)=\z_d^k$ for some~$k=k(\s)$.
Since every coordinate of every element~$\bfxi\in\L$ is a power
of~$\z_d$, we have~$\s(\bfxi)=\xi^k$. Since~$\L$ is a subgroup, it
follows that~$\s(\L)\subseteq\L$, and since~$\s$ is invertible, we see
that~$\s(\L)=\L$. Applying this to $\L=\<\bfzeta_1\>$, we find that
\[
  \bfzeta_2=\s(\bfzeta_1)
  \quad\Longrightarrow\quad
  \<\bfzeta_2\> = \bigl\<\s(\bfzeta_1)\bigr\> = \s\bigl(\<\bfzeta_1\>\bigr)
   = \<\bfzeta_1\>.
\]
This completes the proof that~(b) implies~(a).
\par
For the reverse implication, we assume that
$\<\bfzeta_2\>=\<\bfzeta_1\>$. Write
\[
  \bfzeta_1=(\z_1,\ldots,\z_N)
  \quad\text{with $\z_i$ a primitive $r_i$'th root of unity.}
\]
Then
\[
  r:=\#\<\bfzeta_1\>=\LCM(r_1,\ldots,r_N). 
\]
By assumption, there are exponents~$k$
and~$\ell$ such tha $\bfzeta_2=\bfzeta_1^k$ and $\bfzeta_1=\bfzeta_2^\ell$.
Then $\bfzeta_1^{k\ell}=1$, so
\[
  k\ell\equiv1\pmod{r_i}\quad\text{for all $1\le i\le N$,}
\]
and hence
\[
  k\ell\equiv1\pmod{r}.
\]
In particular, we have $\gcd(k,r)=1$, so there exists an element
\[
  \s\in\Gal\bigl(\QQ(\bfmu_r)/\QQ\bigr)\subset\Gal\bigl(K(\bfmu_r)/K\bigr)
\]
with the property that
\[
  \s(\eta)=\eta^k\quad\text{for every $\eta\in\bfmu_r$.}
\]
(This is where we use the assumption that $K\cap\Qbar=\QQ$ and the
standard fact that
$\Gal\bigl(\QQ(\bfmu_r)/\QQ\bigr)\cong(\ZZ/r\ZZ)^*$.)  The coordinates
of~$\bfzeta_1$ are in~$\bfmu_r$, so we find that
\[
  \s(\bfzeta_1) = \bfzeta_1^k = \bfzeta_2.
\]
This completes the proof that~(a) implies~(b).
\end{proof}

\begin{proof}[Proof of Corollary~$\ref{corollary:WfGfactoroverQ}$]
Lemma~\ref{lemma:cycsubx} tells us that~$\Gal(\Kbar/K)$ acts
transitively on the generators of the cyclic subgroup~$\<\bfxi\>$, so
the formulas in~(a) and~(b) are immediate consequences of
Theorem~\ref{theorem:WfVfprodCf}(b) and~(c).  Then the final statement
follows from Theorem~\ref{theorem:Cfmirred} which tells us that under
our genericity assumption, each~$C_f(\bfxi)$ is irreducible in~~$R$.
\end{proof}

%%%%%%%%%%%%%%%%%%%%%%%%%%%%%%%%%%%%%%%%%%%%%%%%%%%%%%%%%%%%%%%%%%%%%%
\section{Strong Divisibility of Generic DD-Sequences}
\label{section:strongdivgeneric}
%%%%%%%%%%%%%%%%%%%%%%%%%%%%%%%%%%%%%%%%%%%%%%%%%%%%%%%%%%%%%%%%%%%%%%
Classical one-dimensional divisibility sequences over~$\ZZ$, such
as~$a^n-b^n$, Fibonacci and Lucas sequences, and elliptic divisibility
sequences, have the strong divisibility property
\[
  \gcd(W_m,W_n) = \pm W_{\gcd(m,n)},
\]
and it is an exercise to show that such strong divisibility sequences
are automatically divisibility sequences.

\begin{example}
It is easy to construct examples showing that higher dimensional
DD-sequences need not be strong divisibility sequences.  For example,
the DD-sequence associated to $f(X,Y)=X-Y-4$ satisfies
\begin{align*}
  W_4(f) &= 2^{16}\cdot 3\cdot 5^{3}\cdot 13^{2},\\
  W_6(f) &=  2^{18}\cdot 3^{6}\cdot 5^{2}\cdot 7^{5}\cdot 13^{2}\cdot 19^{2}\cdot 31^{2},\\
  \gcd\bigl( W_4(f),W_6(f) \bigr) &= 2^{16} \cdot 3 \cdot 5^2 \cdot 13^2, \\
  W_{\gcd(4,6)}(f) = W_2(f) &= 2^6\cdot 3.
\end{align*}
\end{example}

Our main result in this section says that \emph{generic}
DD-sequences do have the strong divisibility property.

\begin{proposition}
\label{proposition:genericstrongdiv}
Let $\FF$,~$M$,~$R$, and~$f_M$ be as in the statement of
Theorem~$\ref{theorem:Cfmirred}$, so in particular~$R$ is a UFD.
Assume further that~$\#M\ge2$, i.e.,~$f_M$ is not a
monomial. Let~$\L_1$ and~$\L_2$ be finite subgroups
of~$\GG_m^N(\Kbar)$. Then there is an equality of ideals\footnote{The
  quantity~$\gcd\nolimits_R(a,b)$ denotes the largest ideal dividing
  both~$a$ and~$b$. This is well-defined for any UFD, but we note
  that it is not, in general, equal to the ideal generated by~$a$
  and~$b$, the latter being a property of PIDs.}
\[
  \gcd\nolimits_R\bigl(W_{f_M}(\L_1), W_{f_M}(\L_2)\bigr)
  = W_{f_M}(\L_1\cap\L_2)R.
\]
\end{proposition}

The proof uses the following key result.

\begin{lemma}
\label{lemma:VGRVGprimeR}
Continuing the notation from
Proposition~$\ref{proposition:genericstrongdiv}$, let~$\L_1$ and $\L_2$
be finite subgroups of~$\GG_m^N(\Kbar)$. Then
\[
  \L_1\ne\L_2
  \quad\Longrightarrow\quad
  \gcd\nolimits_R\left(V_{f_M}(\L_1), V_{f_M}(\L_2) \right)=1,
\]
where this is an equality of ideals in~$R$, which is a UFD.
\end{lemma}
\begin{proof}
If~$\L_1$, respectively~$\L_2$, is not cyclic, then
Theorem~\ref{theorem:WfprodVf}(a) says that $V_{f_M}(\L_1)$,
respectively~$V_{f_M}(\L_2)$, is~$1$, so the conclusion is
automatically true.

We are thus reduced to the case that $\L_1=\<\bfxi_1\>$ and
$\L_2=\<\bfxi_2\>$. Writing~$n_1$ and~$n_2$ for the orders
of~$\bfxi_1$ and~$\bfxi_2$, respectively,
Theorem~\ref{theorem:WfprodVf}(a) gives
\[
  V_{f_M}(\L_1) = \prod_{i\in(\ZZ/n_1\ZZ)^*} f_M(\bfxi_1^i)
  \quad\text{and}\quad
  V_{f_M}(\L_2) = \prod_{j\in(\ZZ/n_2\ZZ)^*} f_M(\bfxi_2^j).
\]
We prove the contrapositive of the desired statement, so we suppose
that the gcd is larger than~$1$ and aim to prove that~$\L_1=\L_2$. Our
assumption is that $V_{f_M}(\L_1)$ and $V_{f_M}(\L_2)$ have a common
non-trivial factor in~$R$, so they also have a common non-trivial
factor in the UFD $R\otimes_\FF\overline{\FF}$.  But the quantities
$f_M(\bfxi_1^i)$ and $f_M(\bfxi_2^j)$ are irreducible in
$R\otimes_\FF\overline{\FF}$, since they are linear forms in the
variables~$a_\bfm$. It follows that there exists an~$i$ and a~$j$ so
that
\[
  f_M(\bfxi_1^i) = u f_M(\bfxi_2^j) \quad\text{for some unit
    $u\in(R\otimes_\FF\overline{\FF})^*=\overline{\FF}^*$.}
\]
Keeping in mind that the non-zero coefficients~$a_\bfm$ of~$f_M$ are
independent indeterminates and that~$a_\bfzero\ne0$, we first find
that~$u=1$ by comparing the coefficients of~$a_\bfzero$, and then we
find that~$\bfxi_1^i=\bfxi_2^j$ by comparing the coefficients
of~$a_\bfm$ for any non-zero $\bfm\in M(f)$.

We know that $\gcd(i,n_1)=1$, so we can find a $k$ with
$ik\equiv1\pmod{n_1}$. Then
$\bfxi_1=\bfxi_1^{ik}=\bfxi_2^{jk}\in\<\bfxi_2\>$, and similarly using
$\gcd(j,n_2)=1$, we find that $\bfxi_2\in\<\bfxi_1\>$.
Hence~$\<\bfxi_1\>=\<\bfxi_2\>$, i.e,~$\L_1=\L_2$.
\end{proof}

\begin{proof}[Proof of Proposition $\ref{proposition:genericstrongdiv}$]
We compute
\begin{align*}
  \gcd\bigl(W_{f_M}(\L_1),W_{f_M}(\L_2)\bigr)
  &= \gcd\left( \prod_{\L\subseteq\L_1} V_{f_M}(\L),
          \prod_{\L'\subseteq\L_2} V_{f_M}(\L') \right) \\
  &= \prod_{\L\subseteq\L_1~\text{and}~\L\subseteq\L_2} V_{f_M}(\L) 
   \quad\text{from Lemma \ref{lemma:VGRVGprimeR},}\\
  &= \prod_{\L\subseteq\L_1\cap\L_2} V_{f_M}(\L) \\
  &= W_{f_M}(\L_1\cap\L_2).
\end{align*}
This completes the proof of
Proposition~\ref{proposition:genericstrongdiv}.
\end{proof}

%%%%%%%%%%%%%%%%%%%%%%%%%%%%%%%%%%%%%%%%%%%%%%%%%%%%%%%%%%%%%%%%%%%%%%
\section{$\infty$-Growth Properties of DD-Sequences}
\label{section:infgrowth}
%%%%%%%%%%%%%%%%%%%%%%%%%%%%%%%%%%%%%%%%%%%%%%%%%%%%%%%%%%%%%%%%%%%%%%

In this section we consider the growth rate of a DD-sequence~$W_f(\L)$
as a function of~$\|\L\|$. As noted earlier, intuitively we 
expect $\log\bigl|W_f(\L)\bigr|$ to grow like a multiple of~$\|\L\|$,
but there are subtle Diophantine issues at play due to the possibility
of an element~$\bfzeta\in\L$ lying very close to a root of~$f$,
thereby contributing a very small factor~$f(\bfzeta)$
to~$W_f(\L)$. Before stating our main growth conjecture, we need a
number of definitions.

\begin{definition}
Let $G\subseteq\GG_m^N$ be an algebraic subgroup. We let
\[
  \TT(G) := G(\CC) \cap \TT^N 
  = \bigl\{\bfz\in G(\CC) : |z_1|=\cdots=|z_N|=1 \bigr\},
\]
and we let~$\mu_G$ denote normalized Haar measure on the real
torus~$\TT(G)$. Let~$f\in\CC^{(N)}$ be a non-zero Laurent
polynomial. Then the \emph{$G$-Mahler measure of~$f$} is 
\[
  \Mahler_G(f) := 
   \exp\biggl( \int_{\TT(G)} \log\bigl|f(\bfz)\bigr| \,d\mu_G(\bfz)\biggr).
\]
For example, if~$G=\GG_m^N$, then~$\Mahler_G(f)$ is the classical
Mahler measure. To each finite subgroup~$\L\subset\GG_m^N(\CC)$, we
define a measure
\[
  \mu_\L := \frac{1}{\|\L\|}\sum_{\bfzeta\in\L} \d_\bfzeta,
\]
where~$\d_\bfzeta$ denotes a point mass at~$\bfzeta$. We then say that
a collection~$\Lcal$ of finite subgroups of~$\GG_m^N(\CC)$
\emph{converges to~$G$} if there is a weak convergence of measures
\[
  \lim_{\substack{\L\in\Lcal\\ \|\L\|\to\infty\\}} \mu_\L = \mu_G.
\]
\end{definition}

\begin{conjecture}[Growth Conjecture]
\label{conjecture:growth}
Let $G\subseteq\GG_m^N$ be an algebraic subgroup, let~$\Lcal$ be a
collection of finite subgroups of~$\GG_m^N(\CC)$ that converges
to~$G$, and let~$f\in\Qbar^{(N)}$ be a Laurent polynomial
with algebraic coefficients that is not identically zero on~$G$. Then
\[
  \lim_{\substack{\L\in\Lcal\\ \|\L\|\to\infty\\}} 
  \frac{1}{\|\L\|}\log\bigl|W_f(\L)\bigr|
  = \log\Mcal_G(f).
\]
\end{conjecture}

\begin{theorem}
\label{theorem:growth}
With notation as in Conjecture~$\ref{conjecture:growth}$, the conjecture
is true in the following situations.
\begin{parts}
\Part{(a)}
$N=1$.
\Part{(b)}
$N$ is arbitrary and $f$ does not vanish on~$\TT(G)$.
\Part{(c)}
$N\ge2$ and $\Lcal=\{\mu_n^N:n\ge1\}$
and~$f$ is atoral, which we recall means that $\{\bfz\in\TT^N:f(\bfz)=0\}$
has real codimension at least~$2$ in~$\TT^N$.
\end{parts}
\end{theorem}
\begin{proof}
(a) For the convenience of the reader and to illustrate the use of the key
estimate, which is due to Gelfond, we briefly sketch the proof;
cf.\ \cite[Section~7]{MR3082539}.  Factoring
$f(X)=bX^{k}\prod(X-\b_i)$ and using the fact that
$\Mcal(X-\b)=\log\max\bigl\{|\b|,1\bigr\}$, we find that
\[
  \frac{1}{n} \log\bigl|W_n(f)\bigr| - \log \Mcal(f)
  =  \sum_{j=1}^d   \frac{1}{n}
   \log\left(\frac{|\b_j^n-1|}{\max\bigl\{|\b_j^n|,1\bigr\}}\right).
\]
The terms with~$|\b_j|\ne1$ clearly go to~$0$ as~$n\to\infty$, and it
is not hard to see that the same is true for the terms with~$|\b_j|=1$
provided~$\b_j^n$ never gets too close to~$1$. The key to the proof is
thus the following result, which says that $n$'th roots of unity
cannot come too close to the algebraic number~$\b$.

\begin{theorem}
\textup{(Gelfond~\cite{MR0111736})} Let~$\b\in\Qbar^*$ that is not a
root of unity. Then for every~$\e>0$ there is a constant $C(\b,\e)>0$
such that
\[
  |\b^n-1| \ge C(\b,\e)2^{-\e n}
  \quad\text{for all $n\ge1$.}
\]
\end{theorem}  

(We mention that linear forms in logarithms estimates such as those
in~\cite{MR0232736,MR1817252} can be used to prove even stronger results of
the form~$|\b^n-1|\gg n^{-C(\b)}$.)  
\par\noindent(b)\enspace
This is elementary, and indeed is true even if~$f$ has arbitrary complex
coefficients. Our non-vanishing assumption implies that the
function $\log\bigl|f(\bfz)\bigr|$ is continuous on 
the compact set~$\TT(G)$, so the assumed weak convergence of measures
$\mu_\L\to\mu_G$ implies that
\[
  \lim_{\|\L\|\to\infty} \int_{\TT(G)} \log\bigl|f(\bfz)\bigr|\,d\mu_\L(\bfz)
  = \int_{\TT(G)} \log\bigl|f(\bfz)\bigr|\,d\mu_G(\bfz).
\]
But the definition of~$\mu_\L$ as a sum of point masses says that the
left-hand integral is exactly the sum
\[
  \int_{\TT(G)} \log\bigl|f(\bfz)\bigr|\,d\mu_\L(\bfz)
  = \frac{1}{\|\L\|}\sum_{\bfzeta\in\L} \log\bigl|f(\bfzeta)\bigr|
  = \frac{1}{\|\L\|}\log\bigl|W_f(\L)\bigr|,
\]
which gives the desired result.
\par\noindent(c)\enspace
This is due to Lind, Schmidt, and
Verbitskiy~\cite[Theorem~1.3]{MR3082539}. It is likely that their
proof can be adapted to more general~$G$ and~$\Lcal$,
subject to the atoral constraint that $\{\bfz\in\TT(G):f(\bfz)=0\}$ has real
codimension at least~$2$ in~$\TT(G)$.
\end{proof}

\begin{remark}
We mention that if the limit in Conjecture~\ref{conjecture:growth} is
changed to a limsup, then K.\ Schmidt~\cite[Theorem 21.1]{MR1345152}
has shown that
\[
  \limsup_{n\to\infty} \frac{1}{n^N}\log\bigl|W_n(f)\bigr| = \log\Mcal(f),
\]
even if~$f$ is allowed arbitrary complex coefficients. 
\end{remark}

\begin{remark}
\label{remark:cmplxcoef}
On the other hand, it is easy to see that
Conjecture~\ref{conjecture:limMah} is false if~$f$ is allowed to have
complex coefficients, and indeed, it is false in this case even for
atoral~$f$. To construct a counterexample for~$N=1$, let~$\alpha\in\RR$
be a real number that is extremely well approximable by rational numbers.
Set $a=\exp(2\pi i\a)$. Then
\begin{align*}
  W_n(X-a) &\le \min_{0\le k<n} |a-e^{2\pi i k/n}|\cdot 2^{n-1}
      \quad\text{since $|a-\z|\le 2$,}\\
  & \le 2^n\pi \min_{0\le k<n} \left|\a-\frac{k}{n}\right|
      \quad\text{Mean Value Theorem.}
\end{align*}
For an appropriate choice of~$\a$, we can find a sequence of~$k_i/n_i\in\QQ$
satisfying, say,  $|\alpha-k_i/n_i|<2^{-n_i^2}$, and then
\[
  \lim_{i\to\infty} n_i^{-1}\log\bigl|W_{n_i}(X-a)\bigr| = -\infty.
\]
For a counterexample with $N=2$, we can take $f(X,Y)=(X-a)+(Y-1)$,
where now~$\a$ has rational approximations satisfying
$|\alpha-k_i/n_i|<2^{-n_i^3}$. Further, we observe that a linear
polynomial such as~$f$ is always atoral (provided $a\ne1$), since the
intersection of the two circles $\{x-a:x\in\TT^1\}$ and
$\{1-y:y\in\TT^1\}$ is a finite set of points. Hence even
Theorem~\ref{theorem:growth}(c) is false if~$f$ is allowed to have
arbitrary complex coefficients.
\end{remark}

\begin{remark}
Lind has given $f(X,Y)=X+Y+X^{-1}+Y^{-1}-3$ as a specific example of a
polynomial that is not atoral and for which
Conjecture~\ref{conjecture:growth} is currently not known. Here
$\{f=0\}\cap\TT^2$ is an oval containing exactly four points whose
coordinates are roots of unity.
\end{remark}

%%%%%%%%%%%%%%%%%%%%%%%%%%%%%%%%%%%%%%%%%%%%%%%%%%%%%%%%%%%%%%%%%%%%%%
\section{Rank of Apparition for DD-Sequences}
\label{section:rankofapp}
%%%%%%%%%%%%%%%%%%%%%%%%%%%%%%%%%%%%%%%%%%%%%%%%%%%%%%%%%%%%%%%%%%%%%%
Let $f\in R^{(N)}$ be a non-zero Laurent polynomial, let~$\gp$ be a
prime ideal of~$R$, and let~$\L\subset\GG_m^N(\Kbar)$ be a finite
subgroup.  We recall from the introduction that~$\L$ is said to be a
\emph{rank of apparition for~$\gp$} if
\[
   W_f(\L)\in\gp
  \quad\text{and}\quad
   W_f(\L')\notin\gp  \quad\text{for all}\quad \L'\subsetneq\L.
\]
The intuition is that the divisibility of~$W_f(\L)$ by~$\gp$ is not
forced by the fact that~$W_f$ is a divisibility sequence.  The set of
ranks of apparition for~$\gp$ is denoted
\[
  \RA_f(\gp) = \{\L : \text{$\L$ is a rank of apparition for $\gp$}  \}.
\]
We start with an elementary, but useful, result.

\begin{proposition}
\label{proposition:RAfpcyclic}
Let $f\in R^{(N)}$ be a non-zero Laurent polynomial, and let~$\gp$ be a
prime ideal of~$R$.
\begin{parts}
\Part{(a)}
Let $\L\in\RA_f(\gp)$. Then $V_f(\L)\in\gp$.
\Part{(b)}
Every $\L\in\RA_f(\gp)$ is cyclic.
\end{parts}
\end{proposition}
\begin{proof}
(a) Theorem~\ref{theorem:WfprodVf}(c) gives the factorization
\begin{equation}
  \label{eqn:WfLpLpLVf}
  W_f(\L)=\prod_{\L'\subseteq\L} V_f(\L').
\end{equation}
By definition, our assumption that~$\L\in\RA_f(\gp)$ implies
that~$W_f(\L)\in\gp$, so~\eqref{eqn:WfLpLpLVf} tells us
that~$V_f(\L')\in\gp$ for some~$\L'\subseteq\L$. It follows
that~$W_f(\L')\in\gp$, since the analogous factorization of~$W_f(\L')$
contains~$V_f(\L')$ as a factor. By definition of~$\RA_f(\gp)$, we
must have~$\L'=\L$, and hence~$V_f(\L)\in\gp$.
\par\noindent(b)\enspace
Theorem~\ref{theorem:WfprodVf}(a) says that~$V_f(\L)=1$ if~$\L$ is not
cyclic, while~(a) tells us that~$V_f(\L)\in\gp$, so~$V_f(\L)$ cannot
equal~$1$. Hence~$\L$ is cyclic.
\end{proof}

Our main result is an analytic estimate which shows that the set of
ranks of apparition is not too large. It is a generalization of a
theorem of Romanoff~\cite{MR1512916}, as quantified
in~\cite{MR1395936}. Our proof is an adaptation of the proof
in~\cite{MR1395936}.  For ease of exposition, we take~$R=\ZZ$, but
everything easily generalizes to rings of integers in number fields.

\begin{theorem}
\label{theorem:sumlogpp}
Let $f\in\ZZ^{(N)}$ be a non-zero Laurent polynomial. There is a constant 
$C_f$ such that for all $\e>0$,
\[
  \sum_{p~\text{prime}} \frac{\log p}{p}  \sum_{\L\in\RA_f(p)} \frac{1}{|\L|^\e}
  \le (N+1)\e^{-1} + C_f.
\]
\end{theorem}

The proof of Theorem~\ref{theorem:sumlogpp} requires an estimate for
the number of groups of~$\GG_m^N(\CC)$ of given size. The following
result is undoubtedly well-known, but for lack of a suitable
reference, we sketch the proof.

\begin{lemma}
\label{lemma:nuformula}
For $N\ge1$ and $n\ge1$, let
\[
  \nu_N(n) = \#\bigl\{\L\subset\GG_m^N(\CC):\|\L\|=n\bigr\}.
\]
Then for all $k\ge-(N-1)$ we have 
\[
  \sum_{n\le X} n^k \nu_N(n) \sim \frac{X^{N+k}}{N+k}
  \quad\text{as $X\to\infty$.}
\]
\end{lemma}
\begin{proof}[Proof Sketch]
Finite subgroups of~$\GG_m^N(\CC)$ of order~$n$ are dual to
sublattices of~$\ZZ^N$ of index~$n$. The number of the latter is the
degree of the Hecke operator~$T(n)$.  Formal expansions for the
generating function~$\sum T(p^k)X^k$ and the Dirichlet series~$\sum
T(n)n^{-s}$ are given in~\cite[Theorem~3.21]{MR1291394}.  Replacing
each~$T(n)$ in these formulas by its degree, which is~$\nu_N(n)$,
gives (after some manipulation) the beautiful formula
\[
  \sum_{n=1}^\infty \nu_N(n)n^{-s} = \prod_{j=0}^{N-1} \zeta(s-j),
\]
where~$\z(s)$ is the Riemann $\z$-function. Now a standard Tauberian
theorem such as~\cite[Chapter~VI, Section~3]{MR1282723} gives
$\sum_{n\le X} \nu_N(n)n^{-(N-1)} \sim X$, from which it is an
exercise to derive the more general estimate stated in the lemma.
\end{proof}

\begin{proof}[Proof of Theorem~$\ref{theorem:sumlogpp}$]
We set the following useful notation:
\[
  A_f(x) 
  = \prod_{\hidewidth\substack{\L\subseteq\GG_m^N(\CC)\\ \|\L\|\le x\\}\hidewidth} W_f(\L), \quad
  d_f(\L)
  = \sum_{\hidewidth\substack{p~\text{prime}\\ \L\in\RA_f(p)\\}\hidewidth} 
   \frac{\log p}{p}, \quad
  D_f(x) 
  = \sum_{\hidewidth\substack{\L\subseteq\GG_m^N(\CC)\\ \|\L\|\le x\\}\hidewidth} d_f(\L).
\]
We note that if $\L\in\RA_f(p)$ for a lattice with $\|\L\|\le x$, then
$p\mid A_f(x)$, since~$p$ divides the factor~$W_f(\L)$ appearing
in~$A_f(x)$.  We use this observation to estimate
\begin{align*}
  D_f(x) &= \sum_{\substack{\L\subseteq\GG_m^N(\CC)\\ \|\L\|\le x\\}} d_f(\L)
    &&\text{definition of $D_f(x)$,} \\
  &= \sum_{\substack{\L\subseteq\GG_m^N(\CC)\\ \|\L\|\le x\\}} 
         \sum_{\substack{p~\text{prime}\\ \L\in\RA_f(p)\\}} \frac{\log p}{p}
    &&\text{definition of $d_f(\L)$,} \\
  &= \sum_{\substack{p~\text{prime such}\\\text{that $\exists\;\L\in\RA_f(p)$}\\
       \text{with $\|\L\|\le x$}\\}}  \frac{\log p}{p} \\
  &\le \sum_{p\mid A(x)} \frac{\log p}{p}
    &&\text{from above observation,} \\
  &\le \log\log \bigl|A(x)\bigr| + O(1).
\end{align*}
The last inequality is a standard estimate; see for
example~\cite[Section~2]{MR1395936}. 

We define a constant~$C_f'$, depending only on~$f$, by
\[
  C_f' =  \sup_{\substack{(z_1,\ldots,z_N)\in\CC^N\\ |z_1|=\cdots=|z_N|=1\\}}
                \log\bigl|f(z_1,\ldots,z_N)\bigr|.
\]
We use~$C_f'$ to estimate the size of~$A_f(x)$ as follows:
\begin{align}
  \label{eqn:logAfxCfxN1}
  \log \bigl|A_f(x)\bigr| 
  &= \sum_{\substack{\L\subseteq\GG_m^N(\CC)\\ \|\L\|\le x\\}} \log\bigl|W_f(\L)\bigr| 
    &&\text{definition of $A_f(x)$,} \notag \\
  &= \sum_{\substack{\L\subseteq\GG_m^N(\CC)\\ \|\L\|\le x\\}} 
        \sum_{\bfzeta\in\L} \log\bigl|f(\bfzeta)\bigr| 
    &&\text{definition of $W_f(\L)$,} \notag \\
  &\le \sum_{\substack{\L\subseteq\GG_m^N(\CC)\\ \|\L\|\le x\\}} C_f'\|\L\|
    &&\text{definition of $C_f'$,} \notag \\
  &= C_f' \sum_{n\le x} n \nu_N(n)
    &&\text{definition of $\nu_N(n)$,} \notag \\
  &= C_f'  \frac{x^{N+1}}{N+1} \bigl(1+o(1)\bigr)
     &&\text{from Lemma~\ref{lemma:nuformula} with $k=1$,} \notag \\
  &\le C_f'' x^{N+1}
     &&\text{for a new constant.}
\end{align}

We next use a telescoping sum argument (or in fancier terms, Abel
summation), to compute
\begin{align*}
  \sum_{p~\text{prime}}& \frac{\log p}{p}  \sum_{\L\in\RA_f(p)} \frac{1}{\|\L\|^\e}\\
  &= \sum_{\L\subseteq\GG_m^N(\CC)} \frac{1}{\|\L\|^\e} 
         \sum_{\substack{p~\text{prime}\\ \L\in\RA_f(p)\\}} \frac{\log p}{p} \\
  &= \sum_{\L\subseteq\GG_m^N(\CC)} \frac{1}{\|\L\|^\e} \cdot d_f(\L)
    \quad\text{definition of $d_f(\L)$,} \\
  &= \sum_{k=1}^\infty \frac{1}{k^\e} 
      \sum_{\substack{\L\subseteq\GG_m^N(\CC)\\ \|\L\|=k\\}} d_f(\L) \\
  &= \sum_{k=1}^\infty \frac{1}{k^\e} 
     \biggl( \sum_{\substack{\L\subseteq\GG_m^N(\CC)\\ \|\L\|\le k\\}} d_f(\L)
           - \sum_{\substack{\L\subseteq\GG_m^N(\CC)\\ \|\L\|\le k-1\\}} d_f(\L) \biggr) \\
  &= \sum_{k=1}^\infty \frac{1}{k^\e} \Bigl( D_f(k) - D_f(k-1) \Bigr)
    \quad\text{definition of $D_f(x)$,} \\
  &= \biggl(\sum_{k=1}^\infty \frac{1}{k^\e} D_f(k) \biggr)
        - \biggl(\sum_{k=1}^\infty \frac{1}{(k+1)^\e} D_f(k) \biggr) \\
  &\le \sum_{k=1}^\infty \left(\frac{1}{k^\e}-\frac{1}{(k+1)^\e}\right)
        \bigl(\log\log \bigl|A_f(k)\bigr| + O(1)\bigr) \\
    &\le \sum_{k=1}^\infty \left(\frac{1}{k^\e}-\frac{1}{(k+1)^\e}\right)
         \bigl( (N+1)\log(k) + O(1) \bigr) 
     \quad\text{from \eqref{eqn:logAfxCfxN1},} \\
    &\le \sum_{k=1}^\infty 
         \left(\frac{\e}{k^{1+\e}} + O\left(\frac{1}{k^{2+\e}}\right)\right)
         \bigl( (N+1)\log(k) + O(1) \bigr) \\
    &\le (N+1)\e \sum_{k=1}^\infty \frac{\log k}{k^{1+\e}}
             + O\left( \sum_{k=1}^\infty \frac{\e}{k^{1+\e}} \right) 
             + O\left( \sum_{k=1}^\infty \frac{\log k}{k^{2+\e}} \right) \\
    &= (N+1)\e^{-1} + O(1),
\end{align*}
where the $O(1)$ depends on~$f$, but is independent of~$\e$. 
This completes the proof of Theorem~\ref{theorem:sumlogpp}.
\end{proof}

An immediate corollary of Theorem~\ref{theorem:sumlogpp} is an upper
bound for the Dirichlet density of the set of primes such
that~$\RA_f(p)$ contains a ``small'' group. We recall that the
\emph{upper logarithmic Dirichlet density} of a set of primes~$\Pcal$
is the quantity
\[
  \overline{\d}(\Pcal) = \limsup_{s\to1^+} \;(s-1)\sum_{p\in\Pcal}
  \frac{\log p}{p^s}.
\]

\begin{corollary}
\label{corollary:densityRApf}
Let $f\in\ZZ^{(N)}$ be a non-zero Laurent polynomial.  For $\theta>0$,
define
\[
  \Pcal_f(\theta) = \bigl\{p :
    \text{there exists a $\L\in\RA_f(p)$ with $\|\L\|\le p^\theta$} \bigr\}.
\]
Then
\[
  \overline{\d}\bigl(\Pcal_f(\theta)\bigr) \le (N+1)\theta.
\]
\end{corollary}
\begin{proof}
We set $s=1+\e$ and compute
\begin{align*}
  \sum_{p\in\Pcal_f(\theta)} \frac{\log p}{p^s}
  &= \sum_{p\in\Pcal_f(\theta)} \frac{\log p}{p}\cdot\frac{1}{p^\e} \\
  &\le \sum_{p\in\Pcal_f(\theta)} \frac{\log p}{p}\cdot
               \min_{\L\in\RA_f(p)} \frac{1}{\|\L\|^{\e/\theta}}
    \quad\text{definition of $\Pcal_f(\theta)$,} \\
  &\le \sum_{p\in\Pcal_f(\theta)} \frac{\log p}{p}\cdot
               \sum_{\L\in\RA_f(p)} \frac{1}{\|\L\|^{\e/\theta}} 
    \quad\text{adding more $\L$'s,} \\
  &\le \sum_{p} \frac{\log p}{p}\cdot
               \sum_{\L\in\RA_f(p)} \frac{1}{\|\L\|^{\e/\theta}} 
    \quad\text{adding more primes,} \\
  &\le \frac{(N+1)\theta}{\e}+O(1)
    \quad\text{from Theorem~\ref{theorem:sumlogpp}.}
\end{align*}
Multiplying by $s-1=\e$ and letting $s\to1^+$ (so $\e\to0^+$) gives
the desired result. 
\end{proof}

%%%%%%%%%%%%%%%%%%%%%%%%%%%%%%%%%%%%%%%%%%%%%%%%%%%%%%%%%%%%%%%%%%%%%%
\section{Zsigmondy sets of DD-Sequences}
\label{section:zsig}
%%%%%%%%%%%%%%%%%%%%%%%%%%%%%%%%%%%%%%%%%%%%%%%%%%%%%%%%%%%%%%%%%%%%%%

We recall that the \emph{Zsigmondy set} of the DD-sequence~$W_f$ is
the set
\[
  \Zsig(f) := \bigl\{\text{cyclic $\L$} : 
              \text{$W_f(\L)$ has no primitive prime divisors} \bigr\},
\]
where~$\gp$ is a primitive prime divisor for~$\L$ if~$\L\in\RA_f(\gp)$.
Classical results say that the Zsigmondy set is finite for
1-dimensional sequences such as~$a^n-b^n$, Fibonacci and Lucas
sequences, and elliptic divisibility sequences provided that the
sequence has an appropriate growth property. We conjecture a similar
statement for higher dimensional DD-sequences, but the growth
condition is more subtle. Roughly speaking, we want to exclude
those~$\L$ for which the size of~$W_f(\L)$ is not exponential
in~$\|\L\|$.

We recall from Section~\ref{section:infgrowth} that there is a Mahler
measure~$\Mahler_G(f)$ associated to every algebraic
subgroup~$G\subseteq\GG_m^N$. Further, we say that a
collection~$\Lcal$ of finite subgroups of~$\GG_m^N(\CC)$ converges
to~$G$ if there is a weak convergence of measures $\mu_\L\to\mu_G$
as~$\L$ is chosen in~$\Lcal$ with $\|\L\|\to\infty$. (Here~$\mu_G$ is
normalized Haar measure on $G(\CC)\cap\TT^N$ and $\mu_\L$ is
normalized uniform discrete measure on~$\L$.)

\begin{conjecture}
\label{conjecture:zsigfinite}
Let $f\in\Qbar^{(N)}$ be a non-zero Laurent polynomial, let
$G\subset\GG_m^N$ be an algebraic subgroup, and let $\Lcal$ be a
collection of finite cyclic subgroups of~$\GG_m^N(\CC)$ that converges
to~$G$. Then
\[
  \Mcal_G(f) > 1
  \quad\Longrightarrow\quad
  \Zsig(f)\cap\Lcal~\text{is finite.}
\]
\end{conjecture}

As in the classical cases, we expect that a proof of
Conjecture~\ref{conjecture:zsigfinite} will require some version of
the growth conjecture (Conjecture~\ref{conjecture:growth}), a
reasonable description of the sets of ranks of
apparition~$\RA_f(\gp)$, and an estimate showing slow~$\gp$-adic
growth of~$W_f(\L)$ for~$\L$ containing a fixed element
of~$\RA_f(\gp)$.

%%%%%%%%%%%%%%%%%%%%%%%%%%%%%%%%%%%%%%%%%%%%%%%%%%%%%%%%%%%%%%%%%%%%%%
\section{DD-Sequences for Highly Symmetric Polynomials}
\label{section:ddseqsympoly}
%%%%%%%%%%%%%%%%%%%%%%%%%%%%%%%%%%%%%%%%%%%%%%%%%%%%%%%%%%%%%%%%%%%%%%
If a Laurent polynomial has symmetries given by inversions and/or
permutations of its coordinates, then its associated DD-sequence tends
to be powerful, i.e., have many factors that are powers.  In this
section we illustrate this principle for a prototypical highly
symmetric family of polynomials. 

\begin{proposition}
\label{proposition:PTXYWn8}
Let
\[
  P_T(X,Y)=X+X^{-1}+Y+Y^{-1}+T \in \ZZ[T]^{(2)}.
\]
Then the associated DD-sequence of polynomials~$W_n(P_T)\in \ZZ[T]$
factors in~$\ZZ[T]$ as $W_n(P_T)=A_n(T)B_n(T)^8$ with
\[
  \deg B_n(T) = \begin{cases}
        \frac18(n-1)(n-3) & \text{if~$n$ is odd,} \\
        \frac18(n-2)(n-4) & \text{if~$n$ is even.} \\
  \end{cases}
\]
Thus~$W_n(P_T)$, which has degree~$n^2$, is almost an~$8$'th
power.\footnote{One can say even more. If~$n$ is odd, respectively
  even, then $A_n(T)/W_1(P_T)$, respectively $A_n(T)/W_2(P_T)$, is a
  perfect~$4$'th power in~$\ZZ[T]$.}
\end{proposition}

\begin{remark}
The equation~$P_T(X,Y)=0$, which defines a family of elliptic curves
over~$\QQ(T)$, has been much studied.\footnote{The transformation
  $x=-1/XY$, $y=(Y-X)(1+XY)/2X^2Y^2$, maps it to the Weierstrass
  equation $y^2=x^3+(T^2/4-2)x^2+x$.}  For~$t\in\ZZ$, the Mahler
measure~$\Mahler(P_t)$ is conjecturally related to the value
of~$L'(E_t,0)$, and a number of deep relations between
various~$\Mahler(P_t)$ values have been proven, for
example~$\Mahler(P_8)=\Mahler(P_2)^4$
and~$\Mahler(P_5)=\Mahler(P_1)^6$; see~\cite{MR2652904,MR2336636}. It
is thus natural to ask whether~$W_n(P_8)$ and~$W_n(P_2)^4$,
or~$W_n(P_5)$ and~$W_n(P_1)^6$, are similarly related. This was the
original, albeit as yet unsuccessful, motivation for studying the
DD-sequences associated to the family~$P_T(X,Y)$.  However, we can
prove that~$W_n(P_{2T+4})$ and~$W_n(P_T)$ have a common factor
in~$\ZZ[T]$ of degree roughly~$2n$, so in particular~$W_n(P_8)$
and~$W_n(P_2)$ tend to have a fairly large common factor.
\end{remark}

\begin{proposition}
\label{proposition:gcdP2T4PT}
The DD-sequence of polynomials~$W_n(P_T)\in\ZZ[T]$ associated to the
Laurent polynomial~$P_T(X,Y)$ satisfies
\[
  \deg \gcd\nolimits_{\ZZ[T]} \bigl( W_n(P_{2T+4}) , W_n(P_T) \bigr) \ge 
    \begin{cases}
       2n-1 &\text{if $n$ is odd,}\\
       2n-2 &\text{if $n$ is even.}\\
   \end{cases}
\]
%% *** This proposition is true, but there seem to be some exceptional values
%% of n for which the gcd is even larger. Thus
%% n=6, deg(gcd)=14;   n=12, deg(gcd)=30; n=18, deg(gcd)=38
%% So it seems that there is an extra factor if 6 | n.
\end{proposition}

\begin{proof}[Proof of Proposition~$\ref{proposition:PTXYWn8}$]
The maps
\[
  (X,Y)\longmapsto(Y,X)
  \quad\text{and}\quad
  (X,Y)\longmapsto(Y^{-1},X)
\]
generate a group of automorphisms of the 
ring $R^{(2)}=R[X^{\pm1},Y^{\pm1}]$ that is isomorphic to the dihedral
group~$D_4$. Further, the polynomial~$P_T(X,Y)\in\ZZ[T]^{(2)}$ is
fixed by~$D_4$. These maps also induce automorphisms of~$\bfmu_n^2$.
For each~$\bfzeta\in\bfmu_n^2$, we write~$D_4\cdot\bfzeta$ for the
orbit. We are particularly interested in those~$\bfzeta$ whose orbit is
maximal, and we write this set of~$\bfzeta$ as a disjoint union of
orbits, say
\begin{equation}
  \label{eqn:znumDrz8}
  \{\bfzeta\in\bfmu_n^2 : \#(D_4\cdot\bfzeta)=8\}
  = (D_4\cdot\bfzeta_1)\cup (D_4\cdot\bfzeta_2) \cup\cdots\cup
          (D_4\cdot\bfzeta_{k(n)}),
\end{equation}
where later we will give the value of~$k(n)$. We compute
\begin{align*}
  \prod_{\substack{\bfzeta\in\bfmu_n^2\\ \#(D_4\cdot\bfzeta)=8\\}} P_T(\bfzeta)
  &= \prod_{\s\in D_4} \left(\prod_{i=1}^{k(n)} P_T\bigl(\s(\bfzeta_i)\right) 
    \quad\text{from \eqref{eqn:znumDrz8},} \\
  &= \left(\prod_{i=1}^{k(n)} P_T(\bfzeta_i)\right)^8
    \quad\text{since $P_T(X,Y)$ is $D_4$-invariant.}
\end{align*}
On the other hand, the action of~$D_4$ on~$\bfmu_n^2$ commutes with
the action of~$\Gal(\Qbar/\QQ)$, so the set of~$\bfzeta$
satisfying~$\#(D_4\cdot\bfzeta)=8$ is Galois invariant, and hence (since
$P_T(X,Y)$ is $D_4$-invariant), we see that the product
$\prod_{i=1}^{k(n)} P_T(\bfzeta_i)$ is~$\Gal(\Qbar/\QQ)$-invariant. It
is thus in~$\QQ[T]$, and its coefficients are clearly integral
over~$\ZZ$, so it is in~$\ZZ[T]$. This proves that~$W_n(T)$ is
divisible by~$B(T)^8$ for a polynomial~$B(T)\in\ZZ[T]$ of
degree~$k(n)$. It remains to compute~$k(n)$.

Checking the effect of the~$8$ elements of~$D_4$ on~$\bfmu_n^2$, we
find that~$\bfzeta\in\bfmu_n^2$ has a non-trivial stabilizer if and
only if
\[
  \bfzeta\in\bigcup_{\z\in\bfmu_n} 
  \bigl\{(\pm1,\z), (\z,\pm1), (\z,\z), (\z,\z^{-1})\bigr\}.
\]
When~$n$ is odd, this set is the disjoint union of~$(1,1)$ and the
sets
\[
  \bigl\{(1,\z), (\z,1), (\z,\z), (\z,\z^{-1})\bigr\}
\]
with~$\z\in\bfmu_n\setminus1$, so there are~$1+4(n-1)=4n-3$ points in
the set. This gives $k(n)=n^2-4n+3$. When~$n$ is even, a similar
computation, which we leave to the reader, leads to the formula
$k(n)=n^2-6n+8$.
\end{proof}

\begin{proof}[Proof of Proposition~$\ref{proposition:gcdP2T4PT}$]
The key fact is the following identity in the Laurent
ring~$\ZZ[Z^{\pm1}]$:
\[
  P_{2T+4}(Z,Z)
  = 2(Z+Z^{-1})+2T+ 4
  = 2(Z+Z^{-1}+T+2)
  = 2P_T(1,Z).
\]
Further, by exploiting the symmetry of~$P_T$, we obtain
\[
  P_{2T+4}(Z,Z)=P_{2T+4}(Z^{-1},Z)=2P_T(1,Z)=2P_T(Z,1).
\]
Suppose first that~$n$ is odd. Then~$W_n(P_{2T+4})$ has a factor of
the form
\begin{align*}
  P_{2T+4}(1,1)\prod_{1\ne\z\in\bfmu_n} & P_{2T+4}(\z,\z)P_{2T+4}(\z^{-1},\z)\\
  &= 2P_T(1,1)\prod_{1\ne\z\in\bfmu_n} 2P_T(1,\z)2P_T(\z,1)\\
  &= 2^{2n-1}P_T(1,1)\prod_{1\ne\z\in\bfmu_n} P_T(1,\z)P_T(\z,1).
\end{align*}
Other than the~$2^{2n-1}$, this last quantity is also a factor
of~$W_n(P_T)$. Hence $W_n(P_{2T+4})$ and~$W_n(P_T)$ have a common
factor in~$\ZZ[T]$ of degree~$2n-1$.
\par
We obtain a similar result if~$n$ is even, but now we need to keep track
of duplicated factors when~$\z=\pm1$. Thus~$W_n(P_{2T+4})$ has a factor
of the form
\begin{multline*}
  P_{2T+4}(1,1)P_{2T+4}(-1,-1)
  \prod_{\pm1\ne\z\in\bfmu_n} P_{2T+4}(\z,\z)P_{2T+4}(\z^{-1},\z) \\
  =
  2^{2n-2} P_T(1,1)P_T(1,-1)
  \prod_{\pm 1\ne\z\in\bfmu_n} P_T(1,\z)P_T(\z,1),
\end{multline*}
and everything except the~$2^{2n-2}$ is a factor of~$W_n(P_T)$. 
Thus~$W_n(P_{2T+4})$ and~$W_n(P_T)$ have a common
factor in~$\ZZ[T]$ of degree~$2n-2$.
\end{proof}

%%%%%%%%%%%%%%%%%%%%%%%%%%%%%%%%%%%%%%%%%%%%%%%%%%%%%%%%%%%%%%%%%%%%%%
\section{Further Questions}
\label{section:questions}
%%%%%%%%%%%%%%%%%%%%%%%%%%%%%%%%%%%%%%%%%%%%%%%%%%%%%%%%%%%%%%%%%%%%%%

In this section we suggest directions for further research on higher
dimensional DD-sequences. For ease of exposition, we fix a non-zero
Laurent polynomial~$f\in\ZZ^{(N)}$ and consider the
DD-sequence~$\Wcal_f$ associated to~$f$.

\begin{question}
\textup{(\textbf{Zsigmondy})} Is the intersection of the Zsigmondy set
of~$f$ with a set of cyclic subgroups converging to a group~$G$ for
which $\Mahler_G(f)>1$ finite?  See
Conjecture~\ref{conjecture:zsigfinite} in Section~\ref{section:zsig}
for details.
\end{question}

\begin{question}
\textup{(\textbf{Powers and Powerful Numbers})}
\begin{parts}
\Part{(a)}
For a fixed $k\ge2$, when can $\Wcal_f$ contain infinitely many~$k$'th
powers?
\Part{(b)}
When can $\Wcal_f$ contain infinitely many perfect powers?
\Part{(c)}
When can $\Wcal_f$ contain infinitely many powerful numbers?
\end{parts}
\noindent
(The referee has pointed out that~(a) is related to results
of Dvornicich and Zannier~\cite{MR2350852}, and that since Siegel's
theorem can be used for $N=1$, it is possible that Vojta's conjecture
might give some light in higher dimension.)
\end{question}

\begin{question}
\label{question:GMM}
\textup{(\textbf{Growth and Mahler Measure})}
Is it true that
\[
  \lim_{n\to\infty} \frac{1}{n^N}\log\bigl|W_f(n)\bigr| = \log\Mahler(f)?
\]
See Conjecture~\ref{conjecture:growth} in
Section~\ref{section:infgrowth} for details and a more general
version.
\end{question}

\begin{question}
\textup{(\textbf{Order $\boldsymbol\ell$ Ranks of Apparition})}
Let
\[
  \RA_f(p,\ell) := \{\L\in\RA_f(p) : \|\L\|=\ell \}
\]
denote the ranks of apparition for~$p$ of order~$\ell$. Assume that
$N\ge2$ and $M(f)\ge2$.
\begin{parts}
\Part{(a)}
Is it true that for all but finitely many primes~$p$,  the set
\[
  \bigl\{\ell : \RA_f(p,\ell)\ne\emptyset\bigr\}~\text{is infinite?}
\]
\Part{(b)}
Fix a prime~$p$.
Does there exist an~$\e>0$ such that the set
\[
  \bigl\{\ell : \#\RA_f(p,\ell)\ge (1-\e)\#\nu_N(\ell)\bigr\}~\text{has
    density $0$?}
\]
\Part{(c)}
Fix a prime~$p$.  Might it even be true that for all $\e>0$, the set
\[
  \bigl\{\ell : \#\RA_f(p,\ell)\ge \e\#\nu_N(\ell)\bigr\}~\text{has
    density $0$?}
\]
\end{parts}
\end{question}

\begin{question}
\label{question:RAplnumM}
\textup{(\textbf{Ranks of Apparition for Varying $\boldsymbol f$})}
Fix a prime~$\ell$ and a finite set $M\subset\ZZ^N$ of indices
with~$\bfzero\in M$.  Is there a finite set of primes $\Pcal_{\ell,M}$
such for all~$f\in\ZZ^{(N)}$ with shape $M(f)=M$, we have
\[
  \bigl\{ p : \#\RA_f(p,\ell)\ge \#M \bigr\} \subseteq \Pcal_{\ell,M}
  \cup \bigl\{p : M(\tilde f\bmod p)\ne M(f)\bigr\} ?
\]
N.B. The essential content of this question is that~$\Pcal_{\ell,M}$
depends only on the set~$M$, and not on the specific polynomial~$f$.
\end{question}

\begin{remark}
We are able to answer Question~\ref{question:RAplnumM} affirmatively
for~$M=\bigl\{(1,0),(0,1),(0,0)\bigr\}$, i.e., for polynomials of the
form~$f(X,Y)=AX+BY+C$. We omit the rather lengthy case-by-case proof.
\end{remark}

\begin{question}
\textup{(\textbf{Recursion})} Classical divisibility sequences satisfy
recursion formulas, which may be linear (e.g., Fibonacci) or
non-linear (e.g.,~EDS). For 1-dimensional DD-sequences, it is not hard
to prove that~$W_f(n)$ satisfies a linear recursion of order at
most~$2^{\deg(f)}$. If the DD-sequence~$\Wcal_f$ has \emph{true
  dimension\footnote{Roughly, this means that no change of variables
    expresses~$f$ as a monomial times a Laurent polynomial in fewer
    variables. The referee has suggested that the true dimension is
    linked to the stabilizer of the divisor of~$f$ in~$\GG_m^N$, and
    in particular, if this stabilizer is finite, then the true
    dimension is maximal.} $N\ge2$}, is it possible for~$\Wcal_f$ to
satisfy a finite order linear recurrence or an EDS-like non-linear
recurrence?  This could apply to either the partial sequence~$W_f(n)$,
or to the full sequences~$W_f(\L)$, where for the latter one would
first need to formulate a suitable definition of finite order linear
recurrence.  We note that if the growth estimate in
Question~\ref{question:GMM} is valid, then for $N\ge2$ we cannot
have~$W_f(n)=L(n)$ for a linear recurrence~$L$,
since~$\log\bigl|L(n)\bigr|\asymp n$. However, we might ask if it is
possible to have (say) $W_f(n)=L(n^N)$. More generally, how closely
can one approximate the sequence~$W_f(n)$ using a subsequence of a
linear recursion of the form~$L(n^N)$?
\end{question}

\begin{question}
\textup{(\textbf{Signs and Characters})} 
The divisibility property of a DD-sequence is a property of the ideals
generated by the various~$W_f(\L)$, but the sign of~$W_f(\L)$ is also
of interest. More generally, one can look at character values.
\begin{parts}
\Part{(a)}
What can one say about the distribution of the sequence of signs
$\bigl\{\operatorname{sign}W_f(n)\bigr\}$?  Ditto for
$\bigl\{\operatorname{sign}W_f(\L)\bigr\}$?  (See~\cite{MR2226354} for
the analogous question for~EDS.)
\Part{(b)}
Fix a modulus~$q$.  What can one say about the distribution of the
mod~$q$ reduction $\bigl\{W_f(n)\bmod q)\bigr\}$?  Ditto
for~$\bigl\{W_f(\L)\bmod q)\bigr\}$?
\Part{(c)}
More generally, let~$\chi:(\ZZ/q\ZZ)^*\to\CC^*$ be a Dirichlet
character. What can one say about the distribution
of~$\bigl\{\chi\bigl(W_f(n)\bigr)\bigr\}$?  Ditto
for $\bigl\{\chi\bigl(W_f(\L)\bigr)\bigr\}$?
\end{parts}
\end{question}

%%%%%%%%%%%%%%%%%%%%%%%%%%%%%%%%%%%%%%%%%%%%%%%%%%%%%%%%%%%%%%%%%%%%%%
\section{DD-Sequences for Other Groups}
\label{section:ddseqgengp}
%%%%%%%%%%%%%%%%%%%%%%%%%%%%%%%%%%%%%%%%%%%%%%%%%%%%%%%%%%%%%%%%%%%%%%
We briefly indicate how the notion of DD-sequence naturally
generalizes to arbitrary commutative algebraic groups.  More
precisely, let~$\Gcal/R$ be a group scheme over~$R$, let~$\Ocal$
denote the image of the zero section, and let~$D$ be an effective
$R$-divisor on~$\Gcal$. Then a preliminary definition of the associated
DD-sequence~$W_D$ is the sequence
\[
  W_D(n) = (n_*D)\cdot \Ocal,
\]
where~$n:\Gcal\to\Gcal$ is the $n$'th power map, the intersection is
arithmetic intersection on~$\Gcal$, and the resulting
intersection~$W_D(n)$ is naturally identified with an ideal of~$R$
via the map~$\pi_*$ coming from~$\pi:\Gcal\to\Spec(R)$. More generally,
analogously to what we have done for~$\Gcal=\GG_n^N$, we can define
\[
  W_D : \{\text{$R$-isogenies $\f:\Gcal\to\Gcal'$}\}
  \longrightarrow \text{(Ideals of $R$)}
%% \bigl\{\text{finite subgroups $\L$ of $\Gcal(\Kbar)$}\bigr\}
\]
by setting
\begin{equation}
  \label{eqn:genWdf}
  W_D(\f) = (\f_*D)\cdot\Ocal'.
\end{equation}
(Here~$\Gcal'$ may be any $R$-group scheme that admits a finite
$R$-homo\-mor\-phism from~$\Gcal$, and~$\Ocal'$ is the image of the
identity section of~$\Gcal'$.)

For example, if~$\Gcal$ is an elliptic curve~$E$ over~$R$ and~$D=(P)$,
then~$W_D(n)$ is the classical elliptic divisibility sequence
associated to $(E,P)$, and if~$E$ has complex multiplication, then the
more general sequence $W_D(\f)$ is a reformulation of the CM EDS
studied by Streng~\cite{MR2377368}.

We conjecture that generalized DD-sequences exhibit the fast
growth property if their associated Mahler measures are greater than~$1$.
We remark that if instead of the divisor~$D$, we instead used a
point~$P$, then the sequences~$W_P(n)=(n_*P)\cdot\Ocal$ are also quite
interesting (an example being $\gcd(a^n-1,b^n-1)$), but they do not
appear to satisfy the growth property. Similarly, it is not unnatural
to consider sequences of the form $W_{P,D}(n)=(n_*P)\cdot D$. These
sequences probably do have the growth property, but unless~$D$ is of a
very special form, they will not be divisibility sequences. These two
observations may help to justify our use of~\eqref{eqn:genWdf} to
define higher dimensional DD-sequences.

\subsection*{Acknowldegements}
This research was partially supported by Simons Collaboration
  Grant \#241309. 
The author
would like to thank
Dan Bump,
Vesselin Dimitrov,
Sol Friedberg, 
Jeff Hoffstein,
Matilde Lal{\'{\i}}n,
Douglas Lind, 
Chris Smyth,
Cam Stewart,
Andreas Thom,
and
Felipe Voloch
for their helpful advice
(some of which appeared in the answers to the MathOverflow
questions~\cite{MathOverflow98176,MathOverflow178979,MathOverflow46068}).
The author would also like to thank Katherine Stange for her extensive
comments and corrections to an initial draft of this paper, and the
referee for his/her many helpful suggestions and corrections.

%%%%%%%%%%%%%%%%%%%%%%%%%%%%%%%%%%%%%%%%%%%%%%%%%%%%%%%%%%%%%%%%%%%%%%
%% Bibliography
%%%%%%%%%%%%%%%%%%%%%%%%%%%%%%%%%%%%%%%%%%%%%%%%%%%%%%%%%%%%%%%%%%%%%%

%% \bibliographystyle{abbrv}
%% \bibliography{DivSeqHighDim}

\def\cprime{$'$} \def\cprime{$'$}

\end{document}